\documentclass[a4paper,12pt,draft]{article}
\usepackage{amsmath, amsfonts, amssymb, amsthm, graphics}
\usepackage{xcolor}
\newtheorem{theorem}{Theorem}[section]

\newtheorem{proposition}{Proposition}[section]
\newtheorem{corollary}{Corollary}[section]
\newtheorem{definition}{Definition}[section]
\newtheorem{remark}{Remark}[section]

\catcode`@=11 \@addtoreset{equation}{section} \catcode`@=12

\usepackage[normalem]{ulem}

\definecolor{blue-violet}{rgb}{0.54,0.17,0.89}
\definecolor{amethyst}{rgb}{0.6,0.4,0.8}
\definecolor{darkviolet}{rgb}{0.58, 0.0, 0.83}
\definecolor{darkgreen}{rgb}{0,.4,0}
\definecolor{mixedgreen}{rgb}{0.3,0.6,00}
\definecolor{bananayellow}{rgb}{1.0, 0.88, 0.21}
\definecolor{arylideyellow}{rgb}{0.91, 0.84, 0.42}
\definecolor{bananamania}{rgb}{0.98, 0.91, 0.71}

\newcommand{\REMOVEsk}[1]%
           {{\color{magenta}\sout{#1}}}

\allowdisplaybreaks

\begin{document}
\title{Linear Time-Varying Dynamic-Algebraic Equations of Index $\geq 2$ on Time Scales}

\author{Svetlin G. Georgiev\footnote{Sorbonne University, Department of Mathematics, Paris, France}\ and Sergey Kryzhevich\footnote{Gda\'nsk University of Technology, Faculty of Applied Physics and Mathematics, Gda\'nsk, Poland, serkryzh@pg.edu.pl}
}
\maketitle
\begin{abstract}
In this paper, we introduce a class of linear time-varying dynamic-algebraic equations(LTVDAE) of tractability index $\geq 2$ on arbitrary time scales. We propose a procedure for the decoupling of the considered class LTVDAE. In the paper is used a projector approach. This work is a continuation of our previous article where we study equations of index 1.
\end{abstract}


\section{Introduction}

In this paper,  we will investigate the LTVDAE
\begin{equation}
\label{6.1} A^{\sigma}(t)(Bx)^{\Delta}(t)=C^{\sigma}(t) x^{\sigma}(t)+f(t),\quad t\in J,
\end{equation}
where $A: J\to \mathcal{M}_{n\times m}$, $B: J\to \mathcal{M}_{m\times n}$, $C: J\to \mathcal{M}_{n\times n}$,  $f: J\to \mathbb{R}^n$ are given, $x: J\to \mathbb{R}^n$ is the unknown value,  $J\subseteq \mathbb{T}$, $\mathbb{T}$ is a time scale with forward jump operator and delta differentiation operator $\sigma$ and $\Delta$, respectively. Here, with $\mathcal{M}_{p\times q}$ we denote the space of $p\times q$ matrices with real entries.  More precisely, we give conditions for $A$, $B$, $C$, and $f$ under which we construct projectors and matrix chains ensuring decoupling of the LTVDAE \eqref{6.1}. To the best of our knowledge, there are not any investigations devoted to LTVDAE  on arbitrary time scales. Compared to our recent work \cite{SGSK1}, we study a more sophisticated case of index, greater or equal to 2.

The paper is organized as follows. In the next section, we give some basic facts of time scale calculus necessary for our main results. In Section 3, we construct the so-called matrix chains. In Section 4, we deduce some basic properties of the constructed objects. In Section 5, we give a decoupling procedure for LTDAE of index $\geq 2$. In Section 6, we give an example illustrating the obtained results. A conclusion is given in Section 7. Below, we remove the explicit dependence on $t$ for the sake of notational simplicity. 

\section{Time Scales Essential}

In this paper, a time scale denoted by the symbol $\mathbb T$, is any closed non-empty subset of $\mathbb R$. We suppose that a time scale $\mathbb T$ has the topology that inherits from the real numbers with the standard topology. 

We set $\inf\emptyset=\sup{\mathbb T}$ and $\sup\emptyset=\inf{\mathbb T}$.

\begin{definition}
For $t \in {\mathbb T}$ we define the forward jump operator $\sigma : {\mathbb T} \mapsto {\mathbb T}$ as follows
$\sigma(t) = \inf\{s \in {\mathbb T} : s > t\}$, $\rho(t) = \sup\{s \in {\mathbb T} : s < t\}$. We note that $\sigma(t) \ge t$, $\rho(t)\le t$ for any $t \in {\mathbb T}$.
\end{definition}
\begin{definition}
We define the graininess function by the formula $\mu(t)=\sigma(t)-t$. The point is called \emph{right-dense} if $\mu(t)=0$ and right-scattered otherwise. 
\end{definition}
\begin{definition} If $f : {\mathbb T} \mapsto {\mathbb R}$ is a function, then we define
$f^\sigma : {\mathbb T} \mapsto {\mathbb R}$ by
$f^{\sigma}(t)=f(\sigma(t))$ for any $t \in {\mathbb T}$, i.e., $f^\sigma =f\circ \sigma$.
\end{definition}

Here is the definition of a segment of a time scale.

\begin{definition}
Let $a,b\in {\mathbb T}$, $a\le b$. We define the segment $[a,b]$ in ${\mathbb T}$
by $$[a, b] = \{t \in {\mathbb T} : a \le t \le b\}.$$
\end{definition}

Open intervals, half-open intervals, half-lines and  so on are defined accordingly.

\begin{definition}
We say that a function is rd-continuous if it is continuous in all right-dense points and there exists a left limit at left-dense points (that may not coincide with the value of the function at that point).
\end{definition}

Now we introduce the following technical notion. Let ${\mathbb T}^\kappa={\mathbb T}$ if $\sup {\mathbb T}=+\infty$ and ${\mathbb T}^\kappa={\mathbb T} \setminus (\rho(\sup {\mathbb T}), \sup {\mathbb T}]$ otherwise. In other words, this means that the time scale ${\mathbb T}^\kappa$ coincides with the entire scale ${\mathbb T}$ unless $\sup {\mathbb T}$ is an isolated real point. In that case ${\mathbb T}^\kappa$ coincides with all the time scale ${\mathbb T}$ except the point $\sup {\mathbb T}$.

\begin{definition}
Assume that $f : {\mathbb T} \mapsto {\mathbb R}$ is a rd-continuous function and let $t \in {\mathbb T}^\kappa$. We define $f^\Delta(t)$ as follows: for any $\varepsilon > 0$ there is a neighborhood $U$ of $t$, $U=(t-\delta,t+\delta)$ for some
$\delta > 0$, such that
$$|f(\sigma(t))-f(s)-f^\Delta(t)(\sigma(t)-s)|\le \varepsilon|\sigma(t)-s|\quad \mbox{for all}\quad s\in U, s \neq \sigma(t).$$
We say that $f^\Delta(t)$ is the delta or Hilger derivative of $f$ at $t$.

We say that $f$ is delta or Hilger differentiable, shortly differentiable, in $\mathbb T$ if $f^\Delta (t)$ exists for all $t \in {\mathbb T}$. The function $f^\Delta : {\mathbb T} \mapsto {\mathbb R}$ is said to be delta derivative or Hilger derivative, shortly derivative, of $f$ in $\mathbb T$.
\end{definition}

If ${\mathbb T}={\mathbb R}$, then the delta derivative coincides with the classical derivative, if ${\mathbb T} = {\mathbb N}$, this is just the increment $f(n+1)-f(n)$.

We list some basic properties of Hilger derivatives.

\begin{theorem} Let $f :{\mathbb T} \mapsto R$ be a function and let $t\in {\mathbb T}^\kappa$. Then the
following holds.
\begin{enumerate}
\item If $f$ is differentiable at $t$, then $f$ is continuous at $t$.
\item If $f$ is continuous at $t$ and $t$ is right-scattered, then $f$ is differentiable at $t$ with
$$f^\Delta(t)= \dfrac{f(\sigma(t))-f(t)}{\mu(t)}.$$
\item If $t$ is right-dense, then $f$ is differentiable if and only if the limit
$$f^\Delta(t):=\lim_{s\to t} \dfrac{f(t)-f(s)}{t-s}$$
exists and is finite.
\item If $f$ is differentiable at $t$, then
$f (\sigma (t)) = f (t) + \mu(t)f^\Delta (t)$.
\end{enumerate}
\end{theorem}

We list some basic properties of the Hilger derivatives.

\begin{theorem} Assume that $f,g:{\mathbb T}\mapsto {\mathbb R}$ are differentiable at $t\in {\mathbb T}$. Then
\begin{enumerate}
\item the sum $f+g$ is differentiable at $t$ with
$$(f+g)^\Delta(t)=f^\Delta(t)+g^\Delta(t).$$
\item for any constant $\alpha$, the function $\alpha f$ is differentiable at $t$ with $$(\alpha f )^\Delta(t) = \alpha f^\Delta(t).$$
\item if $g(t),g(\sigma(t))\neq 0$ then $f/g$ is differentiable at $t$ with
$$(f/g)^\Delta (t)= \dfrac{f^\Delta(t)g(t)-f(t)g^\Delta(t)}{g(t)g(\sigma(t))}.$$
\item the product $fg$ is differentiable at $t$ with
$$(fg)^\Delta(t) = f^\Delta(t)g(t)+f(\sigma(t))g^\Delta(t)
= f(t)g^\Delta(t)+f^\Delta(t)g(\sigma(t)).$$
\end{enumerate}
\end{theorem}

A symmetric situation may be considered by replacing $\sigma(t)$ with $\rho(t)$. In this case, we can define the left graininess function by the formula: $\nu(t)=t-\rho(t)$. One can defined the scale ${\mathbb T}_\kappa$ similarly to 
${\mathbb T}^\kappa$.

We define the nabla Hilger derivative $f^\nabla(t)$ as follows: for any $\varepsilon > 0$ there is a neighborhood $U$ of $t$, $U=(t-\delta,t+\delta)\bigcap {\mathbb T}$ for some $\delta > 0$, such that
$$|f(\rho(t))-f(s)-f^\Delta(t)(\rho(t)-s)|\le \varepsilon|\rho(t)-s|\quad \mbox{for all}\quad s\in U, s \leq \rho (t).$$

However, one can reduce the case of $\nabla$ derivative to the case of $\delta$ -- derivative by applying the transformation $t\mapsto -t$.

The integral calculus on time scales is also quite well-developed, see \cite{bohner1}. Now we introduce some basic concepts related to the theory of linear systems of the form
\begin{equation}\label{linear}
    x^\Delta=A(t)x
\end{equation}
on time scales.

Properties of solutions of system \eqref{linear} are also studied. In general, they correspond to those for ordinary differential equations with one important exception which can be illustrated by the following example.

\noindent\textbf{Example}. Let ${\mathbb T}={\mathbb N}$. Consider the scalar equation
$x^\Delta=-x$. Then any solution of the considered equation with initial conditions $x(n_0)=x_0$ is zero for any $n>n_0$, so there is no backward uniqueness of solutions. Moreover, solutions with non-zero initial conditions do not exist backwards.

\begin{definition}
We say that the matrix $A$ is regressive with respect to $\mathbb T$ provided $E+\mu(t)A(t)$ is invertible for all $t\in {\mathbb T}$. Similarly, a function $a$ is regressive if $1+\mu(t)a(t)\neq 0$ for all $t\in {\mathbb T}$.
\end{definition}

\begin{theorem} \hfill
\begin{enumerate}
\item Any solution of system \eqref{linear} with initial conditions $x(t_0)=x_0$ exists for any $t\ge t_0$.
\item Solutions of system \eqref{linear} are unique if the matrix is regressive.
\item The matrix-valued function $A$ is regressive if and only if the eigenvalues $\lambda_i(t)$ of $A(t)$ are regressive for all $1 \le  i \le  n$.
\end{enumerate}
\end{theorem}

An analog of the Lagrange method to solve linear non-homogeneous systems holds true for time scale systems, as well.

\section{Matrix Chains}

\label{section3}
Suppose that $A: J\to \mathcal{M}_{n\times m}$, $B:J\to \mathcal{M}_{m\times n}$.
 \begin{definition}
The matrix pair $(A, B)$ is said to be $(\sigma, 1)$-properly stated if $A^{\sigma}$ and $B$ satisfy
\begin{equation}
\label{*1}
 \text{ker}\, A^{\sigma}\oplus \text{im}\,B=\mathbb{R}^m\quad\text{on}\quad   J
 \end{equation}
 and both $\ker A^{\sigma}$ and $\text{im}\,B$ are $\mathcal{C}^1$-smooth spaces.  The condition \eqref{*1} will be called $(\sigma, 1)$-transversality condition. 
 \end{definition}
 Let $(A, B)$ be a $(\sigma, 1)$-properly stated matrix pair. Then $A^{\sigma}$, $B$ and $A^{\sigma}B$ have a constant rank on $J$ and $\ker A^{\sigma}B=\ker B$  on $J$. 
By the $(\sigma, 1)$-transversality condition \eqref{*1} together with the $\mathcal{C}^1$-smoothness of $\ker A^{\sigma}$ and $\text{im}\,B$, it  follows that there is a projector $R\in \mathcal{C}^1(J)$ onto $\text{im}\,B$ and along $\text{ker}\,A^{\sigma}$
(which means that $R|_{\mbox{ker}\, A^\sigma}=0$ and $R|_{\mbox{im}\, B}=\mbox{id}|_{\mbox{im}\, B}$). Denote
$G_0=A^{\sigma}B$ and let $P_0=\Pi_0$ be a continuous projector along $\ker{G}_0$. Set
\begin{eqnarray*}
Q_0&=&M_0= I-\Pi_0,\quad C_0= C,\quad
G_1= G_0+C_0 M_0.
\end{eqnarray*}
\begin{definition} We say that the matrix $B^-$ is the $\{1, 2\}$-inverse of $B$ if
\begin{eqnarray*}
BB^-B&=& B,\quad
B^- BB^-= B^-,\quad
B^-B= \Pi_0,\quad
BB^-= R.
\end{eqnarray*}
\end{definition}
Let $N_0= \text{ker}\,G_0$.
This construction can be iterated for $i\geq 1$ in the following manner.
\begin{description}
\item[(A1)] $G_i$ have a constant rank $r_i$ on $I$.
\item[(A2)] $N_i=\text{ker}\,G_i$ satisfies
\begin{equation*}
(N_0\oplus \cdots\oplus N_{i-1})\cap N_i=\{0\}.
\end{equation*}
\end{description}
We choose a continuous projector $Q_i$ onto $N_i$ such that
\begin{description}
\item[(A3)] $Q_iQ_j=0$, $0\leq j<i$ on $I$.
\end{description}
Set
\begin{equation*}
P_i=I-Q_i,\quad \Pi_i=P_0\ldots P_i,\quad M_i=\Pi_{i-1}-\Pi_i,\quad i\geq 1.
\end{equation*}
and assume
\begin{description}
\item[(A4)] $B\Pi_i B^-\in \mathcal{C}^1(J)$.
\end{description}
Assume that $(A1)$-$(A4)$ hold. We find a matrix $C_i$ so that
\begin{eqnarray*}
C_i^{\sigma}\Pi_i^{\sigma}&=&\left(C_{i-1}^{\sigma}+C_iM_i+G_iB^{-}(B\Pi_i B^-)^{\Delta}B^{\sigma}\right)\Pi_{i-1}^{\sigma},\quad i\in \{1, \ldots, \nu-1\},\\
\mbox{and define}\\
G_i&=& G_{i-1}+C_{i-1}M_{i-1}, \quad i\geq 1.
\end{eqnarray*}
\begin{remark}
Note that the existence of a projector $Q_i$ such that $Q_i Q_j=0$, $0\leq j<i$,
relies on the fact that the condition $(A2)$ makes it possible to choose a projector $Q_i$ onto $N_i$  such that
$N_0\oplus \cdots \oplus N_{i-1}\subseteq \text{ker}\,Q_i$.
Then $\text{ker}\,Q_j\subseteq \text{ker}\,Q_i$, $j\leq i$. Since $Q_j$ projects onto $N_j$, $0\leq j<i$, we have
$\text{im}\,Q_j=N_j$.
Hence, for any $x\in \mathbb{R}^m$ and $0\leq j<i$, we have $Q_j x\in N_j\subseteq \text{ker}\,Q_i$
and then $Q_i Q_j x=0$, i.e., $Q_iQ_j=0$.
\end{remark}
\begin{definition}
Let the matrix pair $(A, B)$ be $(\sigma, 1)$-properly stated matrix pair. Then  the projectors $P_0$ and $Q_0$ are said to be $(\sigma, 1)$-admissible.
\end{definition}
\begin{definition}
A projector sequence $\{Q_0, \ldots, Q_k\}$, respectively $\{P_0, \ldots, P_k\}$, with $k\geq 1$, is said to be $(\sigma, 1)$-preadmissible up to level $k$ if $(A1)$ and $(A2)$ hold for $0\leq i\leq k$ and $(A3)$ holds for $0\leq i<k$. 
\end{definition}
\begin{definition}
A projector sequence $\{Q_0, \ldots, Q_k\}$, respectively $\{P_0, \ldots, P_k\}$, with $k\geq 1$, is said to be $(\sigma, 1)$-admissible up to level $k$ if $(A1)$, $(A2)$  and $(A3)$ hold for $0\leq i\leq k$. \vbox{\index{Admissible Projector Sequence}}
\end{definition}
\begin{definition}
The matrix pair $(A, B)$ is said to be $(\sigma, 1)$-regular with tractability index $\nu\geq 1$ if there exists a $(\sigma, 1)$-admissible projector sequence $\{Q_0, \ldots, Q_{\nu-1}\}$ so that $G_i$, $0\leq i<\nu$, are singular and $G_{\nu}$ is nonsingular $($non-degenerate$)$. 
\end{definition}
\begin{proposition}
\label{proposition3.6} Let $Q_0, \ldots, Q_k$ be a $(\sigma, 1)$-admissible up to level $k$ projector sequence on $I$. Then
\begin{equation*}
\text{ker}\,\Pi_k= N_0\oplus \cdots\oplus N_k.
\end{equation*}
\end{proposition}
\begin{proof}
We prove the statement by induction.
\begin{enumerate}
\item Let $k=1$.
\begin{enumerate}
\item Suppose that $z\in \text{ker}\,P_0 P_1$. Then
\begin{equation}
\label{3.7} P_0P_1z=0.
\end{equation}
Set
\begin{eqnarray*}
z_1&=& P_1z= (I-Q_1)z= z-Q_1z.
\end{eqnarray*}
Hence, \eqref{3.7} holds if and only if $z_1\in \text{ker}\,P_0=N_0$. Note that $Q_1z\in N_1$. Therefore
$z\in N_0\oplus N_1$. Since $z\in \text{ker}\,P_0P_1$ was arbitrarily chosen and we get that it is an element of $N_0\oplus N_1$, we arrive at the relation
\begin{equation}
\label{3.8} \text{ker}\,P_0 P_1\subseteq N_0\oplus N_1.
\end{equation}
\item Let $z\in N_0\oplus N_1$ be arbitrarily chosen. Then
\begin{equation*}
z=z_0+z_1,
\end{equation*}
where $z_0\in N_0$ and $z_1\in N_1$. Since
$\text{ker}\,P_1=N_1$, we have $P_1z_1=0$.
Next, by $z_0\in N_0$, it follows that there is $w_0\in \mathbb{R}^m$ so that
$z_0=Q_0 w_0$. Therefore
\begin{eqnarray*}
P_0P_1z&=& P_0P_1(z_0+z_1)= P_0P_1z_0+P_0P_1z_1= P_0P_1Q_0w_0= P_0(I-Q_1)Q_0 w_0\\ \\
&=& P_0(Q_0-Q_1Q_0)w_0= P_0 Q_0 w_0= 0,
\end{eqnarray*}
i.e., $z\in \text{ker}\,P_0P_1$. Since $z\in N_0\oplus N_1$ was arbitrarily chosen and we get that it is an element of $\text{ker}\,P_0P_1$, we obtain the relation
\begin{equation*}
N_0\oplus N_1\subseteq \text{ker}\,P_0P_1.
\end{equation*}
By the last inclusion and \eqref{3.8}, we obtain
\begin{equation*}
\text{ker}\,P_0P_1=N_0\oplus N_1.
\end{equation*}
\end{enumerate}
\item Assume that
\begin{equation*}
\text{ker}\,P_0P_1\ldots P_i=N_0\oplus N_1\oplus\ldots \oplus N_i
\end{equation*}
for some $i\in \{1, \ldots, k-1\}$.
\item We will prove that
\begin{equation}
\label{3.10} \text{ker}\,P_0P_1\ldots P_{i+1}=N_0\oplus N_1\oplus\ldots \oplus N_{i+1}.
\end{equation}
\begin{enumerate}
\item Let $z\in \text{ker}\,P_0P_1\ldots P_i P_{i+1}$ be arbitrarily chosen. Then
\begin{equation*}
P_0P_1\ldots P_i P_{i+1}z=0
\end{equation*}
and
\begin{eqnarray*}
z_1&:=& P_{i+1}z= (I-Q_{i+1})z= z-Q_{i+1}z.
\end{eqnarray*}
Note that
\begin{equation*}
P_0P_1\ldots P_iz_1=0.
\end{equation*}
Therefore
\begin{eqnarray*}
z_1&\in& \text{ker}\,P_0\ldots P_i= N_0\oplus \ldots \oplus N_i.
\end{eqnarray*}
Next, $Q_{i+1}z\in N_{i+1}$. Consequently
\begin{eqnarray*}
z&=& z_1+Q_{i+1}z\in (N_0\oplus \ldots \oplus N_i)\oplus N_{i+1}= N_0\oplus\ldots \oplus N_{i+1}.
\end{eqnarray*}
Since $z\in \text{ker}\,P_0P_1\ldots P_{i+1}$ was arbitrarily chosen and we obtain that it is an element of $N_0\oplus\ldots \oplus N_{i+1}$, we conclude that
\begin{equation}
\label{3.9} \text{ker}\,P_0\ldots P_{i+1}\subseteq N_0\oplus \ldots N_{i+1}.
\end{equation}
\item $z\in N_0\oplus \ldots \oplus N_{i+1}$ be arbitrarily chosen. Then
\begin{equation*}
z\in (N_0\oplus \ldots \oplus N_i)\oplus N_{i+1}.
\end{equation*}
Hence,
\begin{equation*}
z=z_1+z_2,
\end{equation*}
where
\begin{eqnarray*}
z_1&\in& N_0\oplus \ldots \oplus N_i= \text{ker}\,P_0\ldots P_i\subseteq \ker Q_{i+1}
\end{eqnarray*}
 and $z_2\in N_{i+1}$.  Since $\text{ker}\,P_{i+1}=N_{i+1}$, we get
 \begin{equation*}
 P_{i+1}z_2=0
 \end{equation*}
 and
 \begin{equation*}
 P_0\ldots P_i z_1=0,
 \end{equation*}
 and
 \begin{equation*}
 Q_{i+1}z_1=0.
 \end{equation*}
Therefore,
 \begin{eqnarray*}
 P_0\ldots P_{i+1}z&=& P_0\ldots P_{i+1}(z_1+z_2)= P_0 \ldots P_iP_{i+1}z_1+P_0\ldots P_{i+1}z_2\\ \\
 &=& P_0\ldots P_i P_{i+1}z_1\\ \\
 &=& P_0\ldots P_i(I-Q_{i+1})z_1= P_0\ldots P_i z_1-P_0\ldots P_i Q_{i+1}z_1= 0.
 \end{eqnarray*}
 i.e., $z\in \text{ker}\,P_0\ldots P_{i+1}$.
 Since $z\in N_0\oplus \ldots \oplus N_{i+1}$ was arbitrarily chosen and we get that it is an element of $\text{ker}\,P_0\ldots P_{i+1}$, we find
 \begin{equation*}
 N_0\oplus \ldots \oplus N_{i+1}\subseteq \text{ker}\,P_0\ldots P_{i+1}.
 \end{equation*}
Hence and by \eqref{3.9}, we get \eqref{3.10}. This completes the proof.
\end{enumerate}
\end{enumerate}
\end{proof}

\section{Properties of Matrix Chains}

In this section, we deduce some relations for $M_j$, $\Pi_j$, $P_j$, defined in Section \ref{section3}.
\begin{theorem}
\label{theorem3.95} Let $(A, B)$ be a $(\sigma, 1)$-regular matrix pair with tractability index $\nu$.  Then
\begin{equation}
\label{3.96} M_iM_j=\left\{
\begin{array}{l}
0\quad \text{if}\quad i\ne j\\ \\
M_i\quad \text{if}\quad i=j.
\end{array}
\right.
\end{equation}
\end{theorem}
\begin{proof}
Note that
\begin{equation*}
M_i=\Pi_{i-1}-\Pi_i=\Pi_{i-1}-\Pi_{i-1}P_i=\Pi_{i-1}(I-P_i)=\Pi_{i-1}Q_i.
\end{equation*}
\begin{enumerate}
\item Let $i<j$. Then, using that $Q_lQ_m=0$, $l<m$, and $Q_lP_l=0$, we get
\begin{eqnarray*}
M_iM_j&=& \Pi_{i-1}Q_i \Pi_{j-1}Q_j=\Pi_{i-1}Q_iP_0P_1\ldots P_i\ldots P_{j-1}Q_j=\\ \\
&&\Pi_{i-1}Q_i(I-Q_0)P_1\ldots P_i\ldots P_{j-1}Q_j=\\ \\
&& \Pi_{i-1}Q_i P_1\ldots P_i\ldots P_{j-1}Q_j=\cdots= \Pi_{i-1}Q_i P_i\ldots P_{j-1}Q_j=0. 
\end{eqnarray*}
\item Let $i=j$. Then
\begin{equation*}
M_iM_i=\Pi_{i-1}Q_i\Pi_{i-1}Q_i=\Pi_{i-1}Q_i P_0\ldots P_{i-1}Q_i=\cdots=\Pi_{i-1}Q_i Q_i=\Pi_{i-1}Q_i=M_i.
\end{equation*}
\item Let $i>j$. Then
\begin{equation*}
M_iM_j=\Pi_{i-1}Q_i P_0\ldots P_{j-1}Q_j=\cdots= \Pi_{i-1}Q_i Q_j=0.
\end{equation*}
This completes the proof.
\end{enumerate}
\end{proof}
\begin{theorem}
\label{theorem3.97}Let $(A, B)$ be a $(\sigma, 1)$-regular matrix pair with tractability index $\nu$. Then
\begin{equation*}
\Pi_i \Pi_j=\left\{
\begin{array}{l}
\Pi_i\quad \text{if}\quad i\geq j\\ \\
\Pi_j\quad \text{if}\quad i<j.
\end{array}
\right.
\end{equation*}
\end{theorem}
\begin{proof}\hfill

\begin{enumerate}
\item Let $i=j$. Then
\begin{eqnarray*}
\Pi_i \Pi_i&=&P_0\ldots P_iP_0\ldots P_i=P_0\ldots P_{i-1}(I-Q_i)(I-Q_0)P_1\ldots P_i\\ \\
&=&P_0\ldots P_{i-1}(I-Q_i-Q_0+Q_i Q_0)P_1\ldots P_i\\ \\
&=& P_0\ldots P_{i-1}(P_i-Q_0)P_1\ldots P_i=P_0\ldots P_{i-1}P_i-P_0\ldots P_{i-1}Q_0P_1\ldots P_i\\ \\
&=&\Pi_i-P_0\ldots P_{i-2}(I-Q_i)Q_0 P_1\ldots P_i\\ \\
&=& \Pi_i-P_0\ldots P_{i-2}Q_0P_1\ldots P_i=\cdots= \Pi_i-P_0Q_0P_1\ldots P_i=\Pi_i.
\end{eqnarray*}
\item Let $i>j$. Take $x\in \mathbb{R}^m$ arbitrarily. Suppose that $x\in \text{ker}\,\Pi_j$. Then
$\Pi_j x=0$ and $\Pi_i \Pi_j x=0$. Note that
\begin{eqnarray*}
\text{ker}\,\Pi_j&=& N_0\oplus \cdots \oplus N_j\subseteq N_0\oplus \cdots\oplus N_i= \text{ker}\,\Pi_i.
\end{eqnarray*}
Hence, $x\in \text{ker}\,\Pi_i$ and $\Pi_i x=0$. Therefore
\begin{equation*}
\Pi_i \Pi_j x=\Pi_i x.
\end{equation*}
Let $x\in \text{im}\,\Pi_j$. Then $\Pi_j x=x$ and
\begin{equation*}
\Pi_i \Pi_j x=\Pi_i x.
\end{equation*}
Since $x\in \mathbb{R}^m$ was arbitrarily chosen, we conclude that
\begin{equation*}
\Pi_i \Pi_j =\Pi_i.
\end{equation*}
\item Let $i<j$. Take $x\in \mathbb{R}^m$ arbitrarily. Then
\begin{equation*}
\text{im}\,\Pi_j\subseteq \text{im}\,\Pi_i.
\end{equation*}
Hence,
\begin{equation*}
\Pi_i \Pi_j=\Pi_j.
\end{equation*}
This completes the proof.
\end{enumerate}
\end{proof}
\begin{theorem}
\label{theorem3.99} Let $(A, B)$ be a $(\sigma, 1)$-regular matrix pair with tractability index $\nu$. Then
\begin{equation}
\label{3.100} \Pi_i M_j=\left\{
\begin{array}{l}
0\quad \text{if}\quad i\geq j\\ \\
M_j\quad \text{if}\quad i<j.
\end{array}
\right.
\end{equation}
\end{theorem}
\begin{proof}
We get
\begin{eqnarray*}
\Pi_iM_j&=&\Pi_i(\Pi_{j-1}-\Pi_j)= \Pi_i \Pi_{j-1}-\Pi_i \Pi_j.
\end{eqnarray*}
\begin{enumerate}
\item Let $i\geq j$. Then
\begin{eqnarray*}
\Pi_i M_j&=& \Pi_i-\Pi_i= 0.
\end{eqnarray*}
\item Let $i<j$. Then
\begin{eqnarray*}
\Pi_i M_j&=& \Pi_{j-1}-\Pi_j= M_j.
\end{eqnarray*}
This completes the proof.
\end{enumerate}
\end{proof}
\begin{theorem}
Let $(A, B)$ be a $(\sigma, 1)$-regular matrix pair with tractability index $\nu$. Then
\begin{equation}
\label{3.102} M_i\Pi_j=\left\{
\begin{array}{l}
M_i\quad \text{if}\quad i>j\\ \\
0\quad \text{if}\quad i\leq j.
\end{array}
\right.
\end{equation}
\end{theorem}
\begin{proof}
We find
\begin{eqnarray*}
M_i\Pi_j&=& (\Pi_{i-1}-\Pi_i)\Pi_j= \Pi_{i-1}\Pi_j -\Pi_i \Pi_j.
\end{eqnarray*}
\begin{enumerate}
\item Let $i>j$. Then
\begin{eqnarray*}
M_i \Pi_j&=& \Pi_{i-1}-\Pi_i= M_i.
\end{eqnarray*}
\item Let $i\leq j$. Then
\begin{eqnarray*}
M_i\Pi_j&=& \Pi_j-\Pi_j= 0.
\end{eqnarray*}
This completes the proof.
\end{enumerate}
\end{proof}
\begin{theorem}
\label{theorem3.103} Let $(A, B)$ be a $(\sigma, 1)$-regular matrix pair with tractability index $\nu$. Then
\begin{equation}
\label{3.104} M_i Q_j=\left\{
\begin{array}{l}
0\quad \text{if}\quad i>j\\ \\
M_i\quad \text{if}\quad i=j.
\end{array}
\right.
\end{equation}
\end{theorem}
\begin{proof}
We have
\begin{eqnarray*}
M_i Q_j&=& \Pi_{i-1}Q_iQ_j.
\end{eqnarray*}
\begin{enumerate}
\item Let $i>j$. Then, applying that $Q_i Q_j=0$, 
we arrive at
\begin{eqnarray*}
M_i Q_j&=& \Pi_{i-1}Q_i Q_j= 0.
\end{eqnarray*}
\item Let $i=j$. Then
\begin{eqnarray*}
M_i Q_i&=& \Pi_{i-1}Q_i Q_i= \Pi_{i-1}Q_i= M_i.
\end{eqnarray*}
This completes the proof.
\end{enumerate}
\end{proof}
\begin{corollary}
We have
\begin{equation}
\label{3.104.1}
M_i P_i=0.
\end{equation}
\end{corollary}
\begin{proof}
By Theorem \ref{theorem3.103}, we have
\begin{equation*}
M_i Q_i=M_i.
\end{equation*}
Then
\begin{eqnarray*}
M_i P_i&=& M_i(I-Q_i)= M_i-M_i Q_i= M_i-M_i= 0.
\end{eqnarray*}
This completes the proof.
\end{proof}
\begin{theorem}
\label{theorem3.105} Let $(A, B)$ be a $(\sigma, 1)$-regular matrix pair with tractability index $\nu$. Then
\begin{equation}
\label{3.106} Q_i M_i=Q_i.
\end{equation}
\end{theorem}
\begin{proof}
We have
\begin{eqnarray*}
Q_i M_i&=& Q_i \Pi_{i-1}Q_i= Q_i P_0 P_1\ldots P_{i-1}Q_i=Q_i(I-Q_0)P_1\ldots P_{i-1}Q_i\\ \\
&=& (Q_i -Q_i Q_0)P_1\ldots P_{i-1}Q_i= Q_i P_1\ldots P_{i-1}Q_i=\cdots= Q_i Q_i= Q_i.
\end{eqnarray*}
This completes the proof.
\end{proof}
\begin{theorem}
\label{theorem3.108} Suppose that $(A, B)$ is a $(\sigma, 1)$-regular matrix pair with tractability index $\nu$. Then
\begin{equation}
\label{3.109} Q_jP_i P_{i-1}\ldots P_l=Q_j,\quad j>i>l.
\end{equation}
\end{theorem}
\begin{proof}
We have
\begin{eqnarray*}
Q_jP_iP_{i-1}\ldots P_l&=& Q_j(I-Q_i)P_{i-1}\ldots P_l= (Q_j -Q_j Q_i)P_{i-1}\ldots P_l\\ \\
&=& Q_jP_{i-1}\ldots P_l=\cdots = Q_j P_l=Q_j(I-Q_l)=Q_j -Q_j Q_l=Q_j.
\end{eqnarray*}
This completes the proof.
\end{proof}
\begin{theorem}
\label{theorem3.110} Let $(A, B)$ be a $(\sigma, 1)$-regular matrix pair with tractability index $\nu\geq 1$. Then
\begin{equation}
\label{3.110} P_kP_{k-1}\ldots P_i=I-Q_i-Q_{i+1}-\cdots -Q_k,\quad k\geq i.
\end{equation}
\end{theorem}
\begin{proof}
For $k=i$, we have
\begin{equation*}
P_k=I-Q_k.
\end{equation*}
Let $k>i$. Applying \eqref{3.109}, we get
\begin{eqnarray*}
P_kP_{k-1}\ldots P_i&=& P_k P_{k-1}P_{k-2}\ldots P_i=(I-Q_k)P_{k-1}P_{k-2}\ldots P_i\\ \\
&=& P_{k-1}P_{k-2}\ldots P_i-Q_k P_{k-1}P_{k-2}\ldots P_i= (I-Q_{k-1})P_{k-2}\ldots P_i-Q_k=\cdots\\ \\
&=&  P_i-Q_{i+1}-Q_{i+2}-\cdots -Q_k=I-Q_i-Q_{i+1}-\cdots -Q_k.
\end{eqnarray*}
This completes the proof.
\end{proof}
\begin{theorem}
\label{theorem3.111} Let $(A, B)$ be a $(\sigma, 1)$-regular matrix pair with tractability index $\nu\geq 1$. Then
\begin{equation}
\label{3.112} G_iP_{i-1}=G_{i-1},\quad i\in \{1, \ldots, \nu\}.
\end{equation}
\end{theorem}
\begin{proof}
Fix $i\in \{1, \ldots, \nu\}$. Since
\begin{eqnarray*}
\text{ker}\,M_{i-1}&=& \text{ker}\,Q_{i-1}=\text{im}\,P_{i-1}= \text{im}\,G_{i-1}.
\end{eqnarray*}
Therefore
\begin{equation*}
G_{i-1}Q_{i-1}=0
\end{equation*}
and
\begin{eqnarray*}
G_{i-1}P_{i-1}&=& G_{i-1}(I-Q_{i-1})= G_{i-1}-G_{i-1}Q_{i-1}= G_{i-1}.
\end{eqnarray*}
Now, using that
\begin{equation*}
G_i=G_{i-1}+C_{i-1}Q_{i-1},
\end{equation*}
we find
\begin{eqnarray*}
G_i P_{i-1}&=& (G_{i-1}+C_{i-1}Q_{i-1})P_{i-1}= G_{i-1}P_{i-1}+C_{i-1}Q_{i-1}P_{i-1}= G_{i-1}.
\end{eqnarray*}
This completes the proof.
\end{proof}
\begin{theorem}
\label{theorem3.113} Let $(A, B)$ be a $(\sigma, 1)$-regular  matrix pair with tractability index $\nu\geq 1$. Then
\begin{equation}
\label{3.114} G_0=G_iP_{i-1}\ldots P_0,\quad i\in \{1, \ldots, \nu\}
\end{equation}
and
\begin{equation}
\label{3.115} G_i=G_{\nu}P_{\nu-1}\ldots P_i,\quad i\in \{0, \ldots, \nu\}.
\end{equation}
\end{theorem}
\begin{proof}
We apply \eqref{3.112} and we find
\begin{eqnarray*}
G_0&=& G_1P_0= G_2P_1P_0= G_3P_2P_1P_0=\cdots =G_i P_{i-1}\ldots P_0.
\end{eqnarray*}
Next, for $i\in \{0, \ldots, \nu\}$, we obtain
\begin{eqnarray*}
G_i&=& G_{i+1}P_i= G_{i+2}P_{i+1}P_i=\cdots= G_{\nu}P_{\nu-1}\ldots P_i.
\end{eqnarray*}
This completes the proof.
\end{proof}
\begin{theorem}
\label{theorem3.115} Let $(A, B)$ be a $(\sigma, 1)$-regular matrix pair with tractability index $\nu\geq 1$. Then
\begin{equation}
\label{3.116} G_{\nu}^{-1}G_0=I-Q_0-\cdots - Q_{\nu-1}.
\end{equation}
Moreover,
\begin{equation}
\label{3.117} G_{\nu}^{-1}G_i=I-Q_i-\cdots -Q_{\nu-1},\quad i\in \{0, \ldots, \nu-1\}.
\end{equation}
\end{theorem}
\begin{proof}
By \eqref{3.114}, we find
\begin{equation*}
G_0 =G_{\nu}P_{\nu-1}\ldots P_0,
\end{equation*}
whereupon
\begin{equation*}
G_{\nu}^{-1}G_0=P_{\nu-1}\ldots P_0.
\end{equation*}
Now, using \eqref{3.110}, we find
\begin{equation*}
G_{\nu}^{-1}G_0=I-Q_0-\cdots - Q_{\nu-1}.
\end{equation*}
Next, for $i\in \{0, \ldots, \nu\}$, using \eqref{3.115}, we find
\begin{equation*}
G_i=G_{\nu}P_{\nu-1}\ldots P_i,
\end{equation*}
whereupon, using \eqref{3.110}, we arrive at the equalities
\begin{eqnarray*}
G_{\nu}^{-1}G_i&=& P_{\nu-1}\ldots P_i=I-Q_i-\cdots -Q_{\nu-1}.
\end{eqnarray*}
This completes the proof.
\end{proof}

\section{Decoupling of Dynamic-Algebraic Equations of Index $\geq 2$}

In this section, we consider the equation \eqref{6.1} and suppose that the matrix pair $(A, B)$ is a $(\sigma, 1)$-regular matrix pair with tractability index $\nu\geq 2$. Let $P_0, Q_0, M_0, \Pi_0, C_0, G_0, N_0, B^-$, $\ldots$, $P_{\nu}, Q_{\nu}, M_{\nu}, \Pi_{\nu}, C_{\nu}, G_{\nu}, N_{\nu}$ be as defined in Section \ref{section3}.

\subsection{A Reformulation of previously obtained results} 

Since $R=BB^-$ and $R$ is a continuous projector along $\text{ker}\,A^{\sigma}$, we get
\begin{eqnarray*}
A^{\sigma}&=& A^{\sigma}R= A^{\sigma}BB^{-}= G_0 B^-.
\end{eqnarray*}
Then, we can rewrite the equation \eqref{6.1} as follows
\begin{equation*}
A^{\sigma}BB^{-}(Bx)^{\Delta}=C^{\sigma}x^{\sigma}+f
\end{equation*}
or
\begin{equation*}
G_0B^{-}(Bx)^{\Delta}=C^{\sigma}x^{\sigma}+f.
\end{equation*}
Now, we multiply both sides of the last equation by $G_{\nu}^{-1}$ and find
\begin{equation}
\label{6.11} G_{\nu}^{-1}G_0B^{-}(Bx)^{\Delta}= G_{\nu}^{-1}C^{\sigma}x^{\sigma}+G_{\nu}^{-1}f.
\end{equation}
By \eqref{3.116}, we have
\begin{equation*}
G_{\nu}^{-1}G_0= I-Q_0-\cdots - Q_{\nu-1}.
\end{equation*}
Therefore \eqref{6.11} takes the form
\begin{equation}
\label{4.12} (I-Q_0-\cdots-Q_{\nu-1})B^{-}(Bx)^{\Delta}=G_{\nu}^{-1}C^{\sigma} x^{\sigma}+G_{\nu}^{-1}f.
\end{equation}
Since $\Pi_{\nu-1}$ projects along $N_0\oplus \cdots \oplus N_{\nu-1}$ and $N_i=\text{im}\,Q_i$, we have
\begin{equation*}
\Pi_{\nu-1}(I-Q_0-\cdots-Q_{\nu-1})=\Pi_{\nu-1}
\end{equation*}
and then
\begin{equation*}
B\Pi_{\nu-1}(I-Q_0-\cdots-Q_{\nu-1})=B\Pi_{\nu-1}.
\end{equation*}
Then, we multiply \eqref{4.12} by $B\Pi_{\nu-1}$ and we get
\begin{equation*}
B\Pi_{\nu-1}(I-Q_0-\cdots-Q_{\nu-1})B^{-}(Bx)^{\Delta}=B\Pi_{\nu-1}G_{\nu}^{-1}C^{\sigma} x^{\sigma}+B\Pi_{\nu-1}G_{\nu}^{-1}f
\end{equation*}
or
\begin{equation}
\label{6.13} B\Pi_{\nu-1}B^{-}(Bx)^{\Delta}=B\Pi_{\nu-1}G_{\nu}^{-1}C^{\sigma} x^{\sigma}+B\Pi_{\nu-1}G_{\nu}^{-1}f.
\end{equation}
On the other hand,
\begin{eqnarray*}
\Pi_{\nu-1}B^-B&=& \Pi_{\nu-1}\Pi_0= \Pi_{\nu-1}.
\end{eqnarray*}
Since $B\Pi_{\nu-1}B^-$ and $Bx$ are $\mathcal{C}^1$, we get
\begin{eqnarray*}
B\Pi_{\nu-1}B^{-}(Bx)^{\Delta}&=& (B\Pi_{\nu-1}B^-B x)^{\Delta}-(B\Pi_{\nu-1}B^-)^{\Delta}B^{\sigma}x^{\sigma}=\\ (B\Pi_{\nu-1}x)^{\Delta}-(B\Pi_{\nu-1}B^-)^{\Delta}B^{\sigma}x^{\sigma}.
\end{eqnarray*}
Hence, \eqref{6.13} can be rewritten in the form
\begin{equation*}
(B\Pi_{\nu-1}x)^{\Delta}-(B\Pi_{\nu-1}B^-)^{\Delta}=B\Pi_{\nu-1}G_{\nu}^{-1}C^{\sigma} x^{\sigma}+B\Pi_{\nu-1}G_{\nu}^{-1}f.
\end{equation*}
Now, we decompose $x$ as follows
\begin{equation*}
x=\Pi_{\nu-1}x+(I-\Pi_{\nu-1})x.
\end{equation*}
Then, we find
\begin{eqnarray*}
\lefteqn{(B\Pi_{\nu-1}x)^{\Delta}-(B\Pi_{\nu-1}B^-)^{\Delta}B^{\sigma}(I-\Pi_{\nu-1}^{\sigma}+\Pi_{\nu-1}^{\sigma})x^{\sigma}}\\ \\
&=& B\Pi_{\nu-1}G_{\nu}^{-1}C^{\sigma}(I-\Pi_{\nu-1}^{\sigma}+\Pi_{\nu-1}^{\sigma})x^{\sigma}+B\Pi_{\nu-1}G_{\nu}^{-1}f,
\end{eqnarray*}
or
\begin{equation}
\label{6.14}
\begin{array}{lll}
\lefteqn{(B\Pi_{\nu-1}x)^{\Delta}-(B\Pi_{\nu-1}B^-)^{\Delta}B^{\sigma}\Pi_{\nu-1}^{\sigma}x^{\sigma}-
(B\Pi_{\nu-1}B^-)^{\Delta}B^{\sigma}(I-\Pi_{\nu-1}^{\sigma})x^{\sigma}}\\ \\
&=& B\Pi_{\nu-1}G_{\nu}^{-1}C^{\sigma}\Pi_{\nu-1}^{\sigma}x^{\sigma}+B\Pi_{\nu-1}G_{\nu}^{-1}C^{\sigma}(I-\Pi_{\nu-1}^{\sigma})x^{\sigma}+B\Pi_{\nu-1}G_{\nu}^{-1}f.
\end{array}
\end{equation}
By the definition of $C_j$, we obtain
\begin{eqnarray*}
C_j^{\sigma}\Pi_j^{\sigma}&=& C_{j-1}^{\sigma}\Pi_{j-1}^{\sigma}+\left(C_j M_j+G_jB^-(B\Pi_j B^-)^{\Delta}B^{\sigma}\right)\Pi_{j-1}^{\sigma}\\ \\
&=& \left(C_{j-2}^{\sigma}\Pi_{j-2}^{\sigma}+\left(C_{j-1}M_{j-1}+G_{j-1}B^-(B\Pi_{j-1}B^-)^{\Delta}B^{\sigma}\right)\Pi_{j-2}^{\sigma}\right)\Pi_{j-1}^{\sigma}\\ \\
&&+\left(C_j M_j+G_jB^-(B\Pi_j B^-)^{\Delta}B^{\sigma}\right)\Pi_{j-1}^{\sigma}\\ \\
&=& C_{j-2}^{\sigma}\Pi_{j-2}^{\sigma}\Pi_{j-1}^{\sigma}+\left(C_{j-1} M_{j-1}+G_{j-1}B^-(B\Pi_{j-1} B^-)^{\Delta}B^{\sigma}\right)\Pi_{j-1}^{\sigma}\\ \\
&&+\left(C_j M_j+G_jB^-(B\Pi_j B^-)^{\Delta}B^{\sigma}\right)\Pi_{j-1}^{\sigma}\\ \\
&=&\cdots= C_0^{\sigma}\Pi_0^{\sigma}\ldots \Pi_{j-1}^{\sigma}+\left(\sum\limits_{i=1}^jC_i M_i+\sum\limits_{i=1}^j G_i B^-(B\Pi_i B^-)^{\Delta}B^{\sigma}\right)\Pi_{j-1}^{\sigma}\\ \\
 &=& \left(C^{\sigma}+\sum\limits_{i=1}^jC_i M_i+\sum\limits_{i=1}^j G_i B^-(B\Pi_i B^-)^{\Delta}B^{\sigma}\right)\Pi_{j-1}^{\sigma}.
\end{eqnarray*}
Observe that
\begin{eqnarray*}
\sum\limits_{i=1}^jC_i M_i&=& C_1M_1+C_2M_2+\cdots+C_j M_j\\ \\
&=& G_2-G_1+G_3-G_2+\cdots+G_{j+1}-G_j= G_{j+1}-G_1.
\end{eqnarray*}
Thus,
\begin{equation*}
C_j^{\sigma}\Pi_j^{\sigma}=\left(C^{\sigma}+G_{j+1}-G_1+\sum\limits_{i=1}^j G_iB^{-\sigma}(B\Pi_i B^-)^{\Delta}B^{\sigma}\right)\Pi_{j-1}^{\sigma}.
\end{equation*}
Hence, using that  $\Pi_j M_j= 0$ and $\Pi_{j-1}M_j= M_j$,
we arrive at the following relations
\begin{eqnarray*}
0&=& C_j^{\sigma}\Pi_j^{\sigma} M_j^{\sigma}= \left(C^{\sigma}+G_{j+1}-G_1+\sum\limits_{i=1}^j G_iB^{-}(B\Pi_i B^-)^{\Delta}B^{\sigma}\right)\Pi_{j-1}^{\sigma}M_j^{\sigma}\\ \\
&=& \left(C^{\sigma}+G_{j+1}-G_1+\sum\limits_{i=1}^j G_iB^{-}(B\Pi_i B^-)^{\Delta}B^{\sigma}\right)M_j^{\sigma}\\ \\
&=& C^{\sigma}M_j^{\sigma}+(G_{j+1}-G_1)M_j^{\sigma}+\sum\limits_{i=1}^j G_iB^{-}(B\Pi_i B^-)^{\Delta}B^{\sigma}M_j^{\sigma},
\end{eqnarray*}
which implies
\begin{equation*}
C^{\sigma}M_j^{\sigma}=-(G_{j+1}-G_1)M_j^{\sigma}-\sum\limits_{i=1}^j G_iB^{-}(B\Pi_i B^-)^{\Delta}B^{\sigma}M_j^{\sigma}.
\end{equation*}
Multiplying  the last equation with $G_{\nu}^{-1}$, we get
\begin{eqnarray*}
G_{\nu}^{-1}C^{\sigma}M_j^{\sigma}&=&-G_{\nu}^{-1}(G_{j+1}-G_1)M_j^{\sigma}-\sum\limits_{i=1}^j G_{\nu}^{-1}G_iB^{-}(B\Pi_i B^-)^{\Delta}B^{\sigma}M_j^{\sigma}\\ \\
&=& -(Q_1+\cdots Q_j)M_j^{\sigma}-\sum\limits_{i=1}^j (I-Q_i-\cdots-Q_{\nu-1})B^{-}(B\Pi_i B^-)^{\Delta}B^{\sigma}M_j^{\sigma},
\end{eqnarray*}
i.e.,
\begin{equation}
\label{6.14.1}
\begin{array}{lll}
G_{\nu}^{-1}C^{\sigma}M_j^{\sigma}&=& -(Q_1+\cdots Q_j)M_j^{\sigma}-\sum\limits_{i=1}^j (I-Q_i-\cdots-Q_{\nu-1})B^{-}(B\Pi_i B^-)^{\Delta}B^{\sigma}M_j^{\sigma}.
\end{array}
\end{equation}
Then
\begin{eqnarray*}
&&B\Pi_{\nu-1}G_{\nu}^{-1}C^{\sigma}M_j^{\sigma}\\ \\
&=&  -B\Pi_{\nu-1}(Q_1+\cdots Q_j)M_j^{\sigma}-\sum\limits_{i=1}^jB\Pi_{\nu-1} (I-Q_i-\cdots-Q_{\nu-1})B^{-}(B\Pi_i B^-)^{\Delta}B^{\sigma}M_j^{\sigma}\\ \\
 &=&-\sum\limits_{i=1}^j B\Pi_{\nu-1}B^{-}(B\Pi_i B^-)^{\Delta}B^{\sigma}M_j^{\sigma}\\ \\
&=& -\sum\limits_{i=1}^j (B\Pi_{\nu-1}B^-B\Pi_i B^-)^{\Delta}B^{\sigma}M_j^{\sigma}+ \sum\limits_{i=1}^{j}(B\Pi_{\nu-1}B^-)^{\Delta}B^{\sigma} \Pi_i^{\sigma} B^{-{\sigma}}B^{\sigma} M_j^{\sigma}\\ \\
&=& -\sum\limits_{i=1}^j (B\Pi_{\nu-1}B^-B\Pi_i B^-)^{\Delta}B^{\sigma}M_j^{\sigma}+ \sum\limits_{i=1}^{j-1}(B\Pi_{\nu-1}B^-)^{\Delta}B^{\sigma} \Pi_i^{\sigma} B^{-{\sigma}}B^{\sigma} M_j^{\sigma}\\ \\
&=& -(B\Pi_{\nu-1}B^-)^{\Delta}B^{\sigma} M_j^{\sigma}.
\end{eqnarray*}
By the last equation, we find
\begin{equation}
\label{6.15}B\Pi_{\nu-1}G_{\nu}^{-1}C^{\sigma}\sum\limits_{j=0}^{\nu-1}M_j^{\sigma}=-(B\Pi_{\nu-1}B^-)^{\Delta}B^{\sigma} \sum\limits_{j=0}^{\nu-1}M_j^{\sigma}.
\end{equation}
Observe that
\begin{eqnarray*}
\sum\limits_{j=0}^{\nu-1}M_j^{\sigma}&=& M_0^{\sigma}+M_1^{\sigma}+\cdots +M_{\nu-1}^{\sigma}\\ \\
&=& I-\Pi_0^{\sigma}+\Pi_0^{\sigma} -\Pi_1^{\sigma}+\Pi_1^{\sigma}-\Pi_2^{\sigma}+\cdots+\Pi_{\nu-2}^{\sigma}-\Pi_{\nu-1}^{\sigma}= I-\Pi_{\nu-1}^{\sigma}.
\end{eqnarray*}
From here, applying \eqref{6.15}, we get
\begin{equation*}
B\Pi_{\nu-1}G_{\nu}^{-1}C^{\sigma}(I-\Pi_{\nu-1}^{\sigma})=-(B\Pi_{\nu-1}B^-)^{\Delta}B^{\sigma} (I-\Pi_{\nu-1}^{\sigma}).
\end{equation*}
By the last relation, the equation \eqref{6.14} takes the form
\begin{eqnarray*}
(B\Pi_{\nu-1}x)^{\Delta}-(B\Pi_{\nu-1}B^-)^{\Delta}B^{\sigma}\Pi_{\nu-1}^{\sigma}x^{\sigma}&=& B\Pi_{\nu-1}G_{\nu}^{-1}C^{\sigma}\Pi_{\nu-1}^{\sigma}x^{\sigma}+B\Pi_{\nu-1}G_{\nu}^{-1}f.
\end{eqnarray*}
Set $u=B\Pi_{\nu-1}x$. Then
\begin{eqnarray*}
C^{\sigma}\Pi_{\nu-1}^{\sigma}x^{\sigma}&=& C^{\sigma}\Pi_0^{\sigma} \Pi_{\nu-1}^{\sigma}x^{\sigma}= C^{\sigma}B^{-{\sigma}}B^{\sigma}\Pi_{\nu-1}^{\sigma}x^{\sigma}= C^{\sigma}B^{-{\sigma}} u^{\sigma}.
\end{eqnarray*}
So, we arrive at the equation
\begin{equation*}
u^{\Delta}-(B\Pi_{\nu-1}B^-)^{\Delta}u^{\sigma}=B\Pi_{\nu-1}G_{\nu}^{-1}C^{\sigma}B^{-{\sigma}}u^{\sigma}+
B\Pi_{\nu-1}G_{\nu}^{-1}f,
\end{equation*}
or
\begin{equation}
\label{6.17} u^{\Delta}=(B\Pi_{\nu-1}B^-)^{\Delta}u^{\sigma}+B\Pi_{\nu-1}G_{\nu}^{-1}C^{\sigma}B^{-{\sigma}}u^{{\sigma}}
+B\Pi_{\nu-1}G_{\nu}^{-1}f.
\end{equation}
\begin{definition}
The equation \eqref{6.17} is said to be the inherent equation of the equation \eqref{6.1}. \vbox{\index{Inherent Equation}}
\end{definition}
\begin{theorem}
The subspace $\text{im}\,\Pi_{\nu-1}$ is an invariant subspace for the equation \eqref{6.17}, i.e.,
$u(t_0)\in (\text{im}\,B\Pi_{\nu-1})(t_0)$ for some $t_0\in I$ if and only if $u(t)\in (\text{im}\,B\Pi_{\nu-1})(t)$ 
for any $t\in I$.
\end{theorem}
\begin{proof}
Let $u\in \mathcal{C}^1(I)$ be a solution to equation \eqref{6.17} such that
\begin{equation*}
(B\Pi_{\nu-1})(t_0)u(t_0)=u(t_0).
\end{equation*}
Hence,
\begin{eqnarray*}
u(t_0)&=& (B\Pi_{\nu-1})(t_0)u(t_0)= (B\Pi_{\nu-1}\Pi_0 \Pi_{\nu-1})(t_0)u(t_0)\\ \\
&=& (B\Pi_{\nu-1}B^-B \Pi_{\nu-1})(t_0)u(t_0)= (B\Pi_{\nu-1}B^-)(t_0) (B\Pi_{\nu-1})(t_0) u(t_0) \\ \\
&=& (B\Pi_{\nu-1}B^-)(t_0)u(t_0).
\end{eqnarray*}
We multiply the equation \eqref{6.17} by $I-B\Pi_{\nu-1}B^{-}$ and we get
\begin{eqnarray*}
&& (I-B\Pi_{\nu-1}B^{-})u^{\Delta}\\ \\
&=&(I-B\Pi_{\nu-1}B^{-})(B\Pi_{\nu-1}B^-)^{\Delta}u^{\sigma}+(I-B\Pi_{\nu-1}B^{-})B\Pi_{\nu-1}G_{\nu}^{-1}C^{\sigma}B^{-{\sigma}}u^{\sigma}\\ \\
&&+(I-B\Pi_{\nu-1}B^{-})B\Pi_{\nu-1}G_{\nu}^{-1}f= (I-B\Pi_{\nu-1}B^{-})(B\Pi_{\nu-1}B^-)^{\Delta}u^{\sigma}\\ \\
&&+(B\Pi_{\nu-1}-B\Pi_{\nu-1}B^{-}B\Pi_{\nu-1})G_{\nu}^{-1}C^{\sigma}B^{-{\sigma}}u^{\sigma}\\ \\
&&+(B\Pi_{\nu-1}-B\Pi_{\nu-1}B^{-}B\Pi_{\nu-1})G_{\nu}^{-1}f\\ \\
&=&  (I-B\Pi_{\nu-1}B^{-})(B\Pi_{\nu-1}B^-)^{\Delta}u^{\sigma}+(B\Pi_{\nu-1}-B\Pi_{\nu-1})G_{\nu}^{-1}C^{\sigma}B^{-{\sigma}}u^{\sigma}\\ \\
&&+(B\Pi_{\nu-1}-B\Pi_{\nu-1})G_{\nu}^{-1}f= (I-B\Pi_{\nu-1}B^{-})(B\Pi_{\nu-1}B^-)^{\Delta}u^{\sigma}.
\end{eqnarray*}
Set $v=(I- B\Pi_{\nu-1}B^-)u$. Then
\begin{eqnarray*}
&& v^{\Delta}= \\ \\ 
&&(I-B\Pi_{\nu-1}B^{-})u^{\Delta}+(I-B\Pi_{\nu-1}B^-)^{\Delta}u^{\sigma}\\ \\
&=& (I-B\Pi_{\nu-1}B^{-})(B\Pi_{\nu-1}B^-)^{\Delta}u^{\sigma}+(I-B\Pi_{\nu-1}B^-)^{\Delta}u^{\sigma}\\ \\
&=& \left((I-B\Pi_{\nu-1}B^-)B\Pi_{\nu-1}B^-\right)^{\Delta}u^{\sigma}-\\ \\&& (I-B\Pi_{\nu-1}B^-)^{\Delta}B^{\sigma}\Pi_{\nu-1}^{\sigma}B^{-{\sigma}}u^{\sigma}+ (I-B\Pi_{\nu-1}B^-)^{\Delta}u^{\sigma}\\ \\
&=& (B\Pi_{\nu-1}B^--B\Pi_{\nu-1}B^-B\Pi_{\nu-1}B^-)^{\Delta} u+(I-B\Pi_{\nu-1}B^-)^{\Delta}(I-B^{\sigma}\Pi_{\nu-1}^{\sigma}B^{-{\sigma}})u^{\sigma}\\ \\
&=& (I-B\Pi_{\nu-1}B^{-})^{\Delta}v^{\sigma}.
\end{eqnarray*}
Note that
\begin{eqnarray*}
v(t_0)&=& u(t_0)-(B\Pi_{\nu-1}B^-)(t_0)u(t_0)= u(t_0)-u(t_0)= 0.
\end{eqnarray*}
Therefore, we obtain the following IVP
\begin{eqnarray*}
v^{\Delta}&=&(I-B\Pi_{\nu-1}B^-)^{\Delta}v^{\sigma}\quad \text{on}\quad I,\quad v(t_0)= 0.
\end{eqnarray*}
Therefore $v=0$ on $I$ and then $B\Pi_{\nu-1}B^-u=u$ {on} $I$.
Hence, using that $\text{im}\,B\Pi_{\nu-1}=\text{im}\,B\Pi_{\nu-1}B^-$, we get $B\Pi_{\nu-1}u= u$ {on} $I$.
This completes the proof.
\end{proof}
\subsection{The Component $v_{\nu-1}^{\sigma}$}
\label{section6.4}
Consider the equation \eqref{6.11}.
Note that
\begin{eqnarray*}
M_0+M_1+\cdots+M_{\nu-1}+\Pi_{\nu-1}&=& I-\Pi_0+\Pi_0-\Pi_2+\cdots +\Pi_{\nu-2}-\Pi_{\nu-1}+\Pi_{\nu-1}=I.
\end{eqnarray*}
Then we decompose  the solution $x$ of the equation \eqref{6.11} in the form
\begin{eqnarray*}
x&=& M_0 x+M_1x+\cdots+M_{\nu-1}x+\Pi_{\nu-1}x\\ \\
&=& M_0 x+M_1x+\cdots+M_{\nu-1}x+B^-B \Pi_{\nu-1}x.
\end{eqnarray*}
Set $v_j=M_j x$,  $j\in \{0, \ldots, \nu-1\}$.
We multiply the equation \eqref{6.11} by $M_{\nu-1}$  and we find
\begin{equation}
\label{6.19} M_{\nu-1}G_{\nu}^{-1}G_0B^{-}(Bx)^{\Delta}=M_{\nu-1}G_{\nu}^{-1}C^{\sigma}x^{\sigma}
+M_{\nu-1}G_{\nu}^{-1}f.
\end{equation}
By \eqref{3.116}, we have
\begin{eqnarray*}
M_{\nu-1}G_{\nu}^{-1}G_0&=& M_{\nu-1}(I-Q_0-\cdots -Q_{\nu-1})\\ \\
&=& M_{\nu-1}-M_{\nu-1}Q_0-\cdots-M_{\nu-1}Q_{\nu-1}= M_{\nu-1}-M_{\nu-1}= 0.
\end{eqnarray*}
The equation  \eqref{6.19} can be rewritten in the following manner.
\begin{equation}
\label{6.20} M_{\nu-1}G_{\nu}^{-1}C^{\sigma}x^{\sigma}=-M_{\nu-1}G_{\nu}^{-1}f.
\end{equation}
From here,
\begin{equation}
\label{6.21}
\begin{array}{lll}
M_{\nu-1}G_{\nu}^{-1}f&=& -M_{\nu-1}G_{\nu}^{-1}C^{\sigma}(M_0^{\sigma} x^{\sigma}+M_1^{\sigma}x^{\sigma}+\cdots+M_{\nu-1}^ {\sigma}x+B^{-{\sigma}}B^{\sigma}\Pi_{\nu-1}^{\sigma}x^{\sigma}).
\end{array}
\end{equation}
Using  \eqref{6.14.1}, we arrive at
\begin{eqnarray*}
&& M_{\nu-1}G_{\nu}^{-1}C^{\sigma}M_j^{\sigma}= -M_{\nu-1}(Q_1+\cdots+Q_j)M_j^{\sigma}\\ \\
&&-\sum\limits_{i=1}^jM_{\nu-1}(I-Q_i-Q_{i+1}-\cdots-Q_{\nu-1})B(B\Pi_iB^-)^{\Delta}B^{\sigma}M_j^{\sigma}= 0,\quad j<\nu-1,
\end{eqnarray*}
and
\begin{eqnarray*}
&& M_{\nu-1}G_{\nu}^{-1}C^{\sigma}M_{\nu-1}^{\sigma}= -M_{\nu-1}(Q_1+\cdots+Q_{\nu-1})M_{\nu-1}^{\sigma}\\ \\
&&-\sum\limits_{i=1}^{\nu-1}M_{\nu-1}(I-Q_i-Q_{i+1}-\cdots-Q_{\nu-1})B(B\Pi_iB^-)^{\Delta}B^{\sigma}M_{\nu-1}^{\sigma}= -M_{\nu-1}M_{\nu-1}^{\sigma}.
\end{eqnarray*}
Then, by \eqref{6.21}, we find
\begin{equation*}
-M_{\nu-1}M_{\nu-1}^{\sigma}x^{\sigma}-M_{\nu-1}G_{\nu}^{-1}B^-B \Pi_{\nu-1}x=-M_{\nu-1}G_{\nu}^{-1}f
\end{equation*}
or
\begin{equation}
\label{6.22} M_{\nu-1}v_{\nu-1}^{\sigma}= -M_{\nu-1}G_{\nu}^{-1}C^{\sigma}B^{-{\sigma}} u^{\sigma}+M_{\nu-1}G_{\nu}^{-1}f.
\end{equation}
\subsection{The Operators $U_k$ and $V_k$}
In this section, we define an auxiliary operators $U_k$ and $V_k$ which we apply to simplify the components $v_k^{\sigma}$.
\begin{theorem}
\label{theorem3.121} The operator defined by the formula
\begin{equation}
\label{3.122} V_k=Q_k P_{k+1}\ldots P_{\nu-1}
\end{equation}
is a projector.
\end{theorem}
\begin{proof}
 We get
\begin{eqnarray*}
V_kV_k&=& Q_k P_{k+1}\ldots P_{\nu-1}Q_kP_{k+1}\ldots P_{\nu-1}= Q_k(P_{k+1}\ldots P_{\nu-1}Q_k)P_{k+1}\ldots P_{\nu-1}\\ \\
&=& Q_k Q_k P_{k+1}\ldots P_{\nu-1}= Q_k P_{k+1}\ldots P_{\nu-1}= V_k,
\end{eqnarray*}
i.e., $V_k$ is a projector. This completes the proof.
\end{proof}
\begin{theorem}
\label{theorem3.126} We have
\begin{equation}
\label{3.127} \text{im}\,Q_kP_{k+1}\ldots P_j=N_k,\quad \nu\geq j>k,
\end{equation}
and
\begin{equation}
\label{3.128} \text{ker}\,Q_kP_{k+1}\ldots P_j= N_0\oplus \cdots \oplus N_{k-1}\oplus N_{k+1}\oplus\cdots \oplus N_j\oplus \text{im}\,\Pi_j,\quad \nu\geq j>k.
\end{equation}
\end{theorem}
\begin{proof}
Again, we prove this statement by induction.
\begin{enumerate}
\item Let $j=k+1$.  Since $Q_k$ and $P_{k+1}$ are projectors, we get
\begin{eqnarray*}
\text{im}\,Q_kP_{k+1}&=& \text{im}\,Q_k= N_k.
\end{eqnarray*}
Moreover,
\begin{equation}
\label{3.129}
\begin{array}{lll}
&&\text{ker}\,Q_kP_{k+1}= \text{ker}\,P_{k+1}\oplus (\text{ker}\,Q_k\cap \text{im}\,P_{k+1})\\ \\
&=& N_{k+1}\oplus \left(\left(N_0\oplus\cdots\oplus N_{k-1}\oplus \text{im}\,\Pi_k\right)\cap \left(N_0\oplus\cdots\oplus N_k\oplus \text{im}\,\Pi_{k+1}\right)\right).
\end{array}
\end{equation}
Note that $\text{im}\,\Pi_{k+1}\subseteq \text{im}\,\Pi_k$ and
\begin{equation*}
N_0\oplus \cdots \oplus N_{k-1}\subseteq N_0\oplus \cdots N_{k-1}\oplus N_{k+1}.
\end{equation*}
Since
\begin{equation*}
\left(N_0\oplus \cdots\oplus N_{k-1}\right)\cap N_k=\{0\},
\end{equation*}
we have
\begin{equation*}
\left(N_0\oplus \cdots\oplus N_{k-1}\oplus \text{im}\,\Pi_k\right)\cap N_k=\{0\}.
\end{equation*}
Then
\begin{eqnarray*}
\left(N_0\oplus \cdots\oplus N_{k-1}\oplus \text{im}\,\Pi_{k+1}\right)\cap N_k&\subseteq& \left(N_0\oplus \cdots\oplus N_{k-1}\oplus \text{im}\,\Pi_k\right)\cap N_k=\{0\}.
\end{eqnarray*}
Therefore
\begin{equation*}
\left(N_0\oplus \cdots\oplus N_{k-1}\oplus \text{im}\,\Pi_k\right)\cap \left(N_0\oplus \cdots\oplus N_{k}\oplus \text{im}\,\Pi_{k+1}\right)=\left(N_0\oplus \cdots\oplus N_{k-1}\oplus \text{im}\,\Pi_{k+1}\right).
\end{equation*}
Now, applying \eqref{3.129}, we arrive at
\begin{equation*}
\text{ker}\,Q_k P_{k+1}=N_{k+1}\oplus \left(N_0\oplus \cdots\oplus N_{k-1}\oplus \text{im}\,\Pi_{k+1}\right).
\end{equation*}
\item Assume that $\text{im}\,Q_k P_{k+1}\ldots P_l=N_k$ 
and
\begin{equation*}
\text{ker}\,Q_k P_{k+1}\ldots P_l= N_0\oplus \cdots \oplus N_{k-1}\oplus N_{k+1}\oplus \cdots \oplus N_l\oplus \text{im}\,\Pi_l
\end{equation*}
for some $l>k$, $l<j$.
\item We will prove that $\text{im}\,Q_k P_{k+1}\ldots P_{l+1}=N_k$ and
\begin{equation*}
\text{ker}\,Q_k P_{k+1}\ldots P_{l+1}= N_0\oplus \cdots \oplus N_{k-1}\oplus N_{k+1}\oplus \cdots \oplus N_l\oplus \text{im}\,\Pi_{l+1}.
\end{equation*}
Since $Q_kP_{k+1}\ldots P_l$ {and} $P_{l+1}$ 
are projectors, we obtain
\begin{eqnarray*}
\text{im}\,Q_k P_{k+1}\ldots P_{l+1}&=& \text{im}\,Q_k P_{k+1}\ldots P_l P_{l+1}= \text{im}\,Q_k P_{k+1}\ldots P_l= N_k
\end{eqnarray*}
and
\begin{eqnarray*}
\lefteqn{\text{ker}\,Q_kP_{k+1}\ldots P_{l+1}=\text{ker}\,Q_k P_{k+1}\ldots P_l P_{l+1}}\\ \\
&=& \text{ker}\,P_{l+1}\oplus \left(\left(\text{ker}\,Q_kP_{k+1}\ldots P_l\right)\cap \text{im}\,P_{l+1}\right)\\ \\
&=& N_{l+1}\oplus \left(\left(N_0\oplus \cdots N_{k-1}\oplus N_{k+1}\oplus \cdots \oplus N_l\oplus \text{im}\,\Pi_{l}\right)\cap\left(N_0\oplus\cdots\oplus N_l\oplus \text{im}\,\Pi_{l+1}\right)\right).
\end{eqnarray*}
Note that $\text{im}\,\Pi_{l+1}\subseteq \text{im}\,\Pi_l$ and
\begin{equation*}
N_0\oplus \cdots \oplus N_l\cap N_{l+1}=\{0\}.
\end{equation*}
Then
\begin{equation*}
\left(N_0\oplus \cdots \oplus N_l\oplus \text{im}\,\Pi_{l+1}\right)\cap N_{l+1}=\{0\}
\end{equation*}
and
\begin{eqnarray*}
\lefteqn{\left(N_0\oplus \cdots N_{k-1}\oplus N_{k+1}\oplus \cdots \oplus N_l\oplus \text{im}\,\Pi_{l}\right)\cap\left(N_0\oplus\cdots\oplus N_l\oplus \text{im}\,\Pi_{l+1}\right)}\\ \\
&=& N_0\oplus \cdots \oplus N_{k-1}\oplus N_{k+1}\oplus \cdots \oplus N_l\oplus \text{im}\,\Pi_{l+1}.
\end{eqnarray*}
Therefore
\begin{equation*}
\text{ker}\,Q_kP_{k+1}\ldots P_{l+1}=N_{l+1}\oplus N_0\oplus \cdots \oplus N_{k-1}\oplus N_{k+1}\oplus \cdots \oplus N_l\oplus \text{im}\,\Pi_{l+1}.
\end{equation*}
Hence and by induction, it follows that the assertion holds for any $j>k$, $j\leq \nu$. This completes the proof.
\end{enumerate}
\end{proof}
\begin{theorem}
\label{theorem3.123} We have
\begin{equation}
\label{3.124}\text{im}\,V_k=N_k
\end{equation}
and
\begin{equation}
\label{3.125} \text{ker}\,V_k= N_0\oplus \cdots \oplus N_{k-1}\oplus N_{k+1}\oplus \cdots \oplus N_{\nu-1}\oplus \text{im}\,\Pi_{\nu-1}.
\end{equation}
\end{theorem}
\begin{proof}
By \eqref{3.127}, it follows that
\begin{eqnarray*}
\text{im}\,V_k&=&\text{im}\,Q_kP_{k+1}\ldots P_{\nu-1}=N_k.
\end{eqnarray*}
By \eqref{3.128}, we get
\begin{eqnarray*}
 \text{ker}\,V_k&=&\text{ker}\,Q_kP_{k+1}\ldots P_{\nu-1}= N_0\oplus \cdots \oplus N_{k-1}\oplus N_{k+1}\oplus \cdots \oplus N_{\nu-1}\oplus \text{im}\,\Pi_{\nu-1}.
\end{eqnarray*}
This completes the proof.
\end{proof}
\begin{theorem}
\label{theorem3.130}  Define $U_k=M_kP_{k+1}\ldots P_{\nu-1}$. Then
\begin{equation}
\label{3.131} U_k=\Pi_{k-1}V_k,
\end{equation}
\begin{equation}
\label{3.132} \text{im}\,U_k= \text{im}\,\Pi_{k-1}\cap (N_0\oplus \cdots \oplus N_k)
\end{equation}
and
\begin{equation}
\label{3.133} \text{ker}\,U_k= N_0\oplus \cdots \oplus N_{k-1}\oplus N_{k+1}\oplus \cdots \oplus N_{\nu-1}\oplus \text{im}\,\Pi_{\nu-1}.
\end{equation}
\end{theorem}
\begin{proof}
We have
\begin{eqnarray*}
U_k&=& M_k P_{k+1}\ldots P_{\nu-1}= \Pi_{k-1}Q_kP_{k+1}\ldots P_{\nu-1}= \Pi_{k-1}V_k.
\end{eqnarray*}
Hence, we get
\begin{eqnarray*}
\text{im}\,U_k&=& \text{im}\,\Pi_{k-1}V_k= \text{im}\,\Pi_{k-1}\cap \left(\text{ker}\,\Pi_{k-1}\oplus \text{im}\,V_k\right)= \text{im}\,\Pi_{k-1}\cap(N_0\oplus \cdots \oplus N_{k-1}\oplus N_k)
\end{eqnarray*}
and
\begin{eqnarray*}
\text{ker}\,U_k&=& \text{ker}\,\Pi_{k-1}V_k= \text{ker}\,V_k= N_0\oplus \cdots \oplus N_{k-1}\oplus N_{k+1}\oplus \cdots \oplus N_{\nu-1}\oplus \text{im}\,\Pi_{\nu-1}.
\end{eqnarray*}
This completes the proof.
\end{proof}
\begin{theorem}
\label{theorem3.134} We have
\begin{equation*}
U_k\Pi_i=\left\{
\begin{array}{l}
U_k\quad \text{if}\quad i<k\\ \\
U_k-M_k\quad \text{if}\quad i=k\\ \\
U_k-M_kP_{k+1}\ldots P_i\quad \text{if}\quad i>k.
\end{array}
\right.
\end{equation*}
\end{theorem}

\begin{proof}

\begin{enumerate}
\item Let $i<k$. Then, we find
\begin{eqnarray*}
U_k\Pi_i&=& M_kP_{k+1}\ldots P_{\nu-1}\Pi_i\\ \\
&=& M_k\bigg(P_{k+1}\ldots P_{\nu-1}-Q_i-Q_{i-1}P_i -Q_{i-1}P_{i-1}P_i-\cdots -Q_0 P_1\ldots P_i\bigg)\\ \\
&=& \Pi_{k-1}Q_k\bigg(P_{k+1}\ldots P_{\nu-1}-Q_i-Q_{i-1}P_i -Q_{i-1}P_{i-1}P_i-\cdots -Q_0 P_1\ldots P_i\bigg)\\ \\
&=& \Pi_{k-1}\bigg(Q_kP_{k+1}\ldots P_{\nu-1}-Q_kQ_i-Q_kQ_{i-1}P_i \\ \\
&-& Q_kQ_{i-1}P_{i-1}P_i-\cdots -Q_kQ_0 P_1\ldots P_i\bigg)\\ \\
&=& \Pi_{k-1}Q_k P_{k+1}\ldots P_{\nu-1}=M_kP_{k+1}\ldots P_{\nu-1}=U_k.
\end{eqnarray*}
\item Let $i=k$. Then
\begin{eqnarray*}
U_k\Pi_k&=& M_kP_{k+1}\ldots P_{\nu-1}\Pi_i\\ \\
&=& M_k\bigg(P_{k+1}\ldots P_{\nu-1}-Q_k-Q_{k-1}P_k -Q_{k-1}P_{k-1}P_k-\cdots -Q_0 P_1\ldots P_k\bigg)\\ \\
&=& \Pi_{k-1}Q_k\bigg(P_{k+1}\ldots P_{\nu-1}-Q_k-Q_{k-1}P_i -Q_{k-1}P_{k-1}P_k-\cdots -Q_0 P_1\ldots P_k\bigg)\\ \\
&=& \Pi_{k-1}\bigg(Q_kP_{k+1}\ldots P_{\nu-1}-Q_kQ_k-Q_kQ_{k-1}P_k \\ \\
&-& Q_kQ_{k-1}P_{k-1}P_k-\cdots -Q_kQ_0 P_1\ldots P_k\bigg)\\ \\
&=& \Pi_{k-1}Q_k P_{k+1}\ldots P_{\nu-1}-\Pi_{k-1}Q_k=M_kP_{k+1}\ldots P_{\nu-1}-M_k= U_k-M_k.
\end{eqnarray*}
\item Let $i>k$. Then,  we get
\begin{eqnarray*}
U_k\Pi_i&=& M_kP_{k+1}\ldots P_{\nu-1}\Pi_i\\ \\
&=& M_k\left(P_{k+1}\ldots P_{\nu-1}-Q_0 P_1\ldots P_i-Q_1P_2\ldots P_i-\cdots-Q_k P_{k+1}\ldots P_{\nu-1}\right)\\ \\
&=& \Pi_{k-1}Q_k\left(P_{k+1}\ldots P_{\nu-1}-Q_0 P_1\ldots P_i-Q_1P_2\ldots P_i-\cdots-Q_k P_{k+1}\ldots P_{\nu-1}\right)\\ \\
&=& \Pi_{k-1}\left(Q_kP_{k+1}\ldots P_{\nu-1}-Q_kQ_0 P_1\ldots P_i-Q_kQ_1P_2\ldots P_i\right.\\ \\
&-&\left. \cdots-Q_kQ_k P_{k+1}\ldots P_{\nu-1}\right)\\ \\
&=& \Pi_{k-1}\left(Q_kP_{k+1}\ldots P_{\nu-1}-Q_k P_{k+1}\ldots P_i\right)\\ \\
&=& \Pi_{k-1}Q_k P_{k+1}\ldots P_{\nu-1}-\Pi_{k-1}Q_k P_{k+1}\ldots P_i=U_k-M_kP_{k+1}\ldots P_i.
\end{eqnarray*}
This completes the proof.
\end{enumerate}
\end{proof}
\subsection{The Components $v_k^{\sigma}$}
\label{section6.5}
Since  the projectors $U_k= M_kP_{k+1}\ldots P_{\nu-1}$  project along
\begin{equation*}
N_0\oplus \cdots \oplus N_{k-1}\oplus N_{k+1}\oplus \cdots \oplus N_{\nu-1}\oplus \text{im}\,\Pi_{\nu-1},
\end{equation*}
we have $U_kQ_i=0$, $k\ne i$.  Now, using \eqref{3.116} and then \eqref{3.104}, we obtain
\begin{eqnarray*}
\lefteqn{U_kG_{\nu}^{-1}G_0B^{-}(Bx)^{\Delta}}\\ \\
&=& U_k\left(I-Q_0-\cdots- Q_{\nu-1}\right)B^{-}(Bx)^{\Delta}\\ \\
&=& \left(U_k-U_kQ_0-\cdots - U_kQ_{k-1}-U_kQ_k-U_kQ_{k+1}-\cdots\right.\\ \\
&-& \left.U_kQ_{\nu-1}\right)B^{-}(Bx)^{\Delta}\\ \\
&=& \left(U_k-U_kQ_k\right)B^{-}(Bx)^{\Delta}= U_k\left(I-Q_k\right)B^{-}(Bx)^{\Delta}\\ \\
&=& M_kP_{k+1}\ldots P_{\nu-1}\left(I-Q_k\right)B^{-}(Bx)^{\Delta}\\ \\
&=& M_k\left(P_{k+1}\ldots P_{\nu-1}-P_{k+1}\ldots P_{\nu-1}Q_k\right)B^{-}(Bx)^{\Delta}\\ \\
&=& M_k\left(P_{k+1}\ldots P_{\nu-1}-P_{k+1}\ldots P_{\nu-2}\left(I-Q_{\nu-1}\right)Q_k\right)B^{-}(Bx)^{\Delta}\\ \\
&=& M_k\left(P_{k+1}\ldots P_{\nu-1}-P_{k+1}\ldots P_{\nu-2}\left(Q_k-Q_{\nu-1}Q_k\right)\right)B^{-}(Bx)^{\Delta}\\ \\
&=& M_k\left(P_{k+1}\ldots P_{\nu-1}-P_{k+1}\ldots P_{\nu-2}Q_k\right)B^{-}(Bx)^{\Delta}\\ \\
&=& \cdots= M_k\left(P_{k+1}\ldots P_{\nu-1}-P_{k+1}Q_k\right)B^{-}(Bx)^{\Delta}\\ \\
&=& M_k\left(P_{k+1}\ldots P_{\nu-1}-\left(I-Q_{k+1}\right)\right)B^{-}(Bx)^{\Delta}\\ \\
&=& M_k\left(P_{k+1}\ldots P_{\nu-1}-Q_k+Q_{k+1}Q_k\right)B^{-}(Bx)^{\Delta}\\ \\
&=& M_k\left(P_{k+1}\ldots P_{\nu-1}-Q_k\right)B^{-}(Bx)^{\Delta}\\ \\
&=& \left(M_kP_{k+1}\ldots P_{\nu-1}-M_kAQ_k\right)B^{-}(Bx)^{\Delta}\\ \\
&=& \left(M_kP_{k+1}\ldots P_{\nu-1}-M_k\right)B^{-}(Bx)^{\Delta}=M_k\left(P_{k+1}\ldots P_{\nu-1}-I\right)B^{-}(Bx)^{\Delta},
\end{eqnarray*}
i.e.,
\begin{equation}
\label{4.25} U_kG_{\nu}^{-1}G_0B^{-}(Bx)^{\Delta}=M_k\left(P_{k+1}\ldots P_{\nu-1}-I\right)B^{-}(Bx)^{\Delta}.
\end{equation}
Note that  we have
\begin{eqnarray*}
P_{k+1}\ldots P_{\nu-1}-I&=& -Q_{k+1}-P_{k+1}Q_{k+2}-P_{k+1}P_{k+2}Q_{k+3}-\cdots - P_{k+1}\cdots P_{\nu-1}Q_{\nu-1}.
\end{eqnarray*}
Since $Q_j M_j= Q_j$, we get
\begin{eqnarray*}
P_{k+1}\ldots P_{\nu-1}-I&=& -Q_{k+1}M_{k+1}-P_{k+1}Q_{k+2}M_{k+2}-P_{k+1}P_{k+2}Q_{k+3}M_{k+3}\\ \\
&&-\cdots- P_{k+1}\ldots P_{\nu-2}Q_{\nu-1}M_{\nu-1}.
\end{eqnarray*}
Then, \eqref{4.25} takes the form
\begin{equation}
\label{6.26}
\begin{array}{lll}
U_kG_{\nu}^{-1}G_0B^{-}(Bx)^{\Delta}&=& M_k\bigg( -Q_{k+1}M_{k+1}-P_{k+1}Q_{k+2}M_{k+2}-P_{k+1}P_{k+2}Q_{k+3}M_{k+3}\\ \\
&&-\cdots- P_{k+1}\ldots P_{\nu-2}Q_{\nu-1}M_{\nu-1}\bigg).
\end{array}
\end{equation}
For $j>0$, we have
\begin{equation*}
BM_j x= BM_j B^- Bx
\end{equation*}
and then
\begin{eqnarray*}
(BM_j x)^{\Delta}&=& (BM_j B^-Bx)^{\Delta}= (BM_j B^-)^{\Delta}B^{\sigma}x^{\sigma}+BM_jB^{-}(Bx)^{\Delta}
\end{eqnarray*}
or
\begin{equation}
\label{6.27} BM_jB^{-}(Bx)^{\Delta}=(BM_j x)^{\Delta}-(BM_J B^-)^{\Delta}B^{\sigma}x^{\sigma}.
\end{equation}
Since $P_0=B^-B$ and
\begin{eqnarray*}
M_j&=& P_0 M_j= B^-BM_j,
\end{eqnarray*}
by \eqref{6.26}, we find
\begin{eqnarray*}
&&U_kG_{\nu}^{-1}G_0B^{-}(Bx)^{\Delta}\\ \\
&=& M_k\bigg(-Q_{k+1}B^{-}BM_{k+1}B^{-}(Bx)^{\Delta}-P_{k+1}Q_{k+2}B^{-}BM_{k+2}B^{-}(Bx)^{\Delta}\\ \\
&&-P_{k+1}P_{k+2}Q_{k+3}B^{-}BM_{k+3}B^{-}(Bx)^{\Delta}\cdots\\ \\
&& -P_{k+1}\ldots P_{\nu-2} Q_{\nu-1}B^{-}BM_{\nu-1}B^{-}(Bx)^{\Delta}\bigg)\\ \\
&=& M_k\bigg(-Q_{k+1}B^{-}\left(Bv_{k+1}\right)^{\Delta}-P_{k+1} Q_{k+2}B^{-}\left(Bv_{k+2}\right)^{\Delta}\\ \\
&&-P_{k+1}P_{k+2}Q_{k+3}B^{-}\left(Bv_{k+3}\right)^{\Delta}\cdots\\ \\
&&-P_{k+1}\ldots P_{\nu-1}Q_{\nu-1}B^{-}\left(Bv_{\nu-1}\right)^{\Delta}\bigg)\\ \\
&&+M_k\bigg(Q_{k+1}B^{-}\left(BM_{k+1}B^-\right)^{\Delta}B^{\sigma}x^{\sigma}+P_{k+1}Q_{k+2}B^{-}\left(BM_{k+2}B^-\right)^{\Delta}B^{\sigma}x^{\sigma}\\ \\
&&+P_{k+1}P_{k+2}Q_{k+3}B^{-}\left(BM_{k+3}B^-\right)^{\Delta}B^{\sigma}x^{\sigma}\cdots\\ \\
&&+P_{k+1}\ldots P_{\nu-2}Q_{\nu-1}B^{-}\left(BM_{\nu-1}B^-\right)^{\Delta}B^{\sigma}x^{\sigma}\bigg).
\end{eqnarray*}
Set
\begin{eqnarray*}
\mathcal{N}_{kk+1}&=& -M_kQ_{k+1}B^{-},\quad \mathcal{N}_{kk+2}= -M_kP_{k+1}Q_{k+2}B^{-},\\ \\
\mathcal{N}_{kk+3}&=& -M_kP_{k+1}P_{k+2}Q_{k+3}B^{-},\quad \cdots\\ \\
\mathcal{N}_{k\nu-1}&=& -M_kP_{k+1}\ldots P_{\nu-2}Q_{\nu-1}B^{-}.
\end{eqnarray*}
Therefore
\begin{eqnarray*}
&&U_kG_{\nu}^{-1}G_0B^{-}(Bx)^{\Delta} \\ \\
&=& \sum\limits_{j=k+1}^{\nu-1}\mathcal{N}_{kj}(Bv_j)^{\Delta}\\ \\
&&+M_k\bigg(Q_{k+1}B^{-}\left(BM_{k+1}B^-\right)^{\Delta}Bx+P_{k+1}Q_{k+2}B^{-}\left(BM_{k+2}B^-\right)^{\Delta}B^{\sigma}x^{\sigma}\\ \\
&&+P_{k+1}P_{k+2}Q_{k+3}B^{-}\left(BM_{k+3}B^-\right)^{\Delta}B^{\sigma}x^{\sigma}
\\ \\ &&+\cdots
+P_{k+1}\ldots P_{\nu-2}Q_{\nu-1}B^{-}\left(BM_{\nu-1}B^-\right)^{\Delta}B^{\sigma}x^{\sigma}\bigg).
\end{eqnarray*}
Set
\begin{eqnarray*}
J&=& M_k\bigg(Q_{k+1}B^{-}\left(BM_{k+1}B^-\right)^{\Delta}Bx+P_{k+1}Q_{k+2}B^{-}\left(BM_{k+2}B^-\right)^{\Delta}B^{\sigma}x^{\sigma}\\ \\
&&+P_{k+1}P_{k+2}Q_{k+3}B^{-}\left(BM_{k+3}B^-\right)^{\Delta}B^{\sigma}x^{\sigma}+\cdots\\ \\
&&+P_{k+1}\ldots P_{\nu-2}Q_{\nu-1}B^{-}\left(BM_{\nu-1}B^-\right)^{\Delta}B^{\sigma}x^{\sigma}\bigg).
\end{eqnarray*}
Note that $BM_0=0$ and
\begin{eqnarray*}
Bx&=& B\Pi_{\nu-1}x+B\sum\limits_{j=0}^{\nu-1}M_j x= B\Pi_{\nu-1}x+B\sum\limits_{j=1}^{\nu-1}M_j x.
\end{eqnarray*}
Then $J$ takes the form
\begin{eqnarray*}
J&=& M_k\bigg(Q_{k+1}B^{-}\left(BM_{k+1}B^-\right)^{\Delta}+P_{k+1}Q_{k+2}B^{-}\left(BM_{k+2}B^-\right)^{\Delta}\\ \\
&&+P_{k+1}P_{k+2}Q_{k+3}B^{-}\left(BM_{k+3}B^-\right)^{\Delta}+\cdots\\ \\
&&+P_{k+1}\ldots P_{\nu-2}Q_{\nu-1}B^{-}\left(BM_{\nu-1}B^-\right)^{\Delta}\bigg)B^{\sigma}\Pi_{\nu-1}^{\sigma}x^{\sigma}\\ \\
&&+M_k\bigg(Q_{k+1}B^{-}\left(BM_{k+1}B^-\right)^{\Delta}+P_{k+1}Q_{k+2}B^{-}\left(BM_{k+2}B^-\right)^{\Delta}\\ \\
&&+P_{k+1}P_{k+2}Q_{k+3}B^{-}\left(BM_{k+3}B^-\right)^{\Delta}\cdots\\ \\
&&+P_{k+1}\ldots P_{\nu-2}Q_{\nu-1}B^{-}\left(BM_{\nu-1}B^-\right)^{\Delta}\bigg)\sum\limits_{j=0}^{\nu-1}B^{\sigma}M_j^{\sigma}x^{\sigma}.
\end{eqnarray*}
Denote
\begin{eqnarray*}
J_1&=& M_k\bigg(Q_{k+1}B^{-}\left(BM_{k+1}B^-\right)^{\Delta}+P_{k+1}Q_{k+2}B^{-}\left(BM_{k+2}B^-\right)^{\Delta}\\ \\
&&+P_{k+1}P_{k+2}Q_{k+3}B^{-}\left(BM_{k+3}B^-\right)^{\Delta}+\cdots\\ \\
&&+P_{k+1}\ldots P_{\nu-2}Q_{\nu-1}B^{-}\left(BM_{\nu-1}B^-\right)^{\Delta}\bigg)B^{\sigma}\Pi_{\nu-1}^{\sigma}x^{\sigma},\\ \\
J_2&=& M_k\bigg(Q_{k+1}B^{-}\left(BM_{k+1}B^-\right)^{\Delta}+P_{k+1}Q_{k+2}B^{-}\left(BM_{k+2}B^-\right)^{\Delta}\\ \\
&&+P_{k+1}P_{k+2}Q_{k+3}B^{-}\left(BM_{k+3}B^-\right)^{\Delta}+\cdots\\ \\
&&+P_{k+1}\ldots P_{\nu-2}Q_{\nu-1}B^{-}\left(BM_{\nu-1}B^-\right)^{\Delta}\bigg)\sum\limits_{j=1}^{\nu-1}B^{\sigma}M_j^{\sigma}x^{\sigma}.
\end{eqnarray*}
We have
\begin{eqnarray*}
M_j B^-B\Pi_{\nu-1}x&=& M_j P_0 \Pi_{\nu-1}x= M_j \Pi_{\nu-1}x= 0.
\end{eqnarray*}
Hence,
\begin{eqnarray*}
&&(BM_j B^-)^{\Delta}B^{\sigma} \Pi_{\nu-1}^{\sigma}x^{\sigma}\\ \\&=& -BM_jB^{-}(B\Pi_{\nu-1}x)^{\Delta}+(BM_j B^-B \Pi_{\nu-1}x)^{\Delta}= -BM_jB^{-}(B\Pi_{\nu-1}x)^{\Delta}.
\end{eqnarray*}
Therefore, for $J_1$ we get the following representation.
\begin{eqnarray*}
J_1&=& M_k\bigg(Q_{k+1}B^{-}\left(BM_{k+1}B^-\right)^{\Delta}+P_{k+1}Q_{k+2}B^{-}\left(BM_{k+2}B^-\right)^{\Delta}\\ \\
&&+P_{k+1}P_{k+2}Q_{k+3}B^{-}\left(BM_{k+3}B^-\right)^{\Delta}+\cdots\\ \\
&&+P_{k+1}\ldots P_{\nu-2}Q_{\nu-1}B^{-}\left(BM_{\nu-1}B^-\right)^{\Delta}\bigg)B^{\sigma}\Pi_{\nu-1}^{\sigma}x^{\sigma}\\ \\
&=& -M_k\bigg(Q_{k+1}B^{-}BM_{k+1}B^{-}+P_{k+1}Q_{k+2}B^{-}BM_{k+2}B^{-}\\ \\
&&+ P_{k+1}P_{k+2}Q_{k+3}B^{-}BM_{k+3}B^{-}+\cdots \\ \\
&&+P_{k+1}\ldots P_{\nu-2}Q_{\nu-1}B^{-}BM_{\nu-1} B^{-}\bigg)(B\Pi_{\nu-1}x)^{\Delta}B^{\sigma}\Pi_{\nu-1}^{\sigma}x^{\sigma}\\ \\
&=& -M_k\bigg(Q_{k+1}M_{k+1}+P_{k+1}Q_{k+2}M_{k+2}+P_{k+1}P_{k+2}Q_{k+3}M_{k+3}+\cdots\\ \\
&&+ P_{k+1}\ldots P_{\nu-2}Q_{\nu-1}M_{\nu-1}\bigg)B^{-}(B\Pi_{\nu-1}x)^{\Delta}B^{\sigma}\Pi_{\nu-1}^{\sigma}x^{\sigma}\\ \\
&=& -M_k\bigg(Q_{k+1}+P_{k+1}Q_{k+2}+P_{k+1}P_{k+2}Q_{k+3}+\cdots\\ \\
&&+ P_{k+1}\ldots P_{\nu-2}Q_{\nu-1}\bigg)B^{-}(B\Pi_{\nu-1}x)^{\Delta}B^{\sigma}\Pi_{\nu-1}^{\sigma}x^{\sigma}\\ \\
&=& M_k\left(P_{k+1}\ldots P_{\nu-1}-I\right)B^{-}(B\Pi_{\nu-1}x)^{\Delta}B^{\sigma}\Pi_{\nu-1}^{\sigma}x^{\sigma}\\ \\
&=& M_k\left(P_{k+1}\ldots P_{\nu-1}-I\right)B^{-}(B\Pi_{\nu-1}x)^{\Delta}B^{\sigma}B^{-\sigma}B^{\sigma}\Pi_{\nu-1}^{\sigma}x^{\sigma},
\end{eqnarray*}
i.e.,
\begin{eqnarray*}
J_1&=&  M_k\left(P_{k+1}\ldots P_{\nu-1}-I\right)B^{-}(B\Pi_{\nu-1}x)^{\Delta}B^{\sigma}B^{-\sigma}B^{\sigma}\Pi_{\nu-1}^{\sigma}x^{\sigma}.
\end{eqnarray*}
Set
\begin{equation*}
\widetilde{\mathcal{K}}_k= M_k\left(P_{k+1}\ldots P_{\nu-1}-I\right)B^{-}(B\Pi_{\nu-1}x)^{\Delta}B^{\sigma}.
\end{equation*}
Thus,
\begin{equation*}
J_1= \widetilde{\mathcal{K}}_k B^{-\sigma}u^{\sigma}.
\end{equation*}
Next,
\begin{eqnarray*}
J_2&=& M_k\bigg(Q_{k+1}B^{-}\left(BM_{k+1}B^-\right)^{\Delta}+P_{k+1}Q_{k+2}B^{-}\left(BM_{k+2}B^-\right)^{\Delta}\\ \\
&&+P_{k+1}P_{k+2}Q_{k+3}B^{-}\left(BM_{k+3}B^-\right)^{\Delta}+\cdots\\ \\
&&+P_{k+1}\ldots P_{\nu-2}Q_{\nu-1}B^{-}\left(BM_{\nu-1}B^-\right)^{\Delta}\bigg)\sum\limits_{j=1}^{\nu-1}B^{\sigma}M_j^{\sigma}x^{\sigma}\\ \\
&=& M_k\bigg(Q_{k+1}B^{-}\left(BM_{k+1}B^-\right)^{\Delta}+P_{k+1}Q_{k+2}B^{-}\left(BM_{k+2}B^-\right)^{\Delta}\\ \\
&&+P_{k+1}P_{k+2}Q_{k+3}B^{-}\left(BM_{k+3}B^-\right)^{\Delta}+\cdots\\ \\
&&+P_{k+1}\ldots P_{\nu-2}Q_{\nu-1}B^{-}\left(BM_{\nu-1}B^-\right)^{\Delta}\bigg)B^{\sigma}\sum\limits_{j=1}^{\nu-1}M_j^{\sigma}B^{-\sigma}B^{\sigma}M_j^{\sigma}x^{\sigma}.
\end{eqnarray*}
Define
\begin{eqnarray*}
\widetilde{\mathcal{M}}_{kj}&=& -M_k\bigg(Q_{k+1}B^{-}\left(BM_{k+1}B^-\right)^{\Delta}+P_{k+1}Q_{k+2}B^{-}\left(BM_{k+2}B^-\right)^{\Delta}\\ \\
&&+P_{k+1}P_{k+2}Q_{k+3}B^{-}\left(BM_{k+3}B^-\right)^{\Delta}+\cdots\\ \\
&&+P_{k+1}\ldots P_{\nu-2}Q_{\nu-1}B^{-}\left(BM_{\nu-1}B^-\right)^{\Delta}\bigg)B^{\sigma}M_j^{\sigma}B^{-\sigma}B^{\sigma}.
\end{eqnarray*}
Therefore
\begin{equation*}
J_2= \sum\limits_{j=1}^{\nu-1}\widetilde{\mathcal{M}}_{kj}M_j^{\sigma} x^{\sigma}.
\end{equation*}
Note that
\begin{eqnarray*}
(BM_j B^-)^{\Delta}B^{\sigma}M_j^{\sigma} B^{-\sigma}&=& -BM_jB^{-}(BM_jB^-)^{\Delta}+(BM_j B^-BM_j B^-)^{\Delta}\\ \\ &=& (BM_jB^-)^{\Delta}-BM_jB^{-}(BM_j B^-)^{\Delta}
\end{eqnarray*}
and
\begin{eqnarray*}
(BM_i B^-)^{\Delta}B^{\sigma}M_j^{\sigma} B^{-\sigma}&=& (BM_i B^-BM_j B^-)^{\Delta}-BM_iB^{-}(BM_j B^-)^{\Delta}\\ \\
&=&-BM_iB^{-}(BM_j B^-)^{\Delta},\quad i\ne j.
\end{eqnarray*}
Hence, for $j>k$, we have
\begin{eqnarray*}
\widetilde{\mathcal{M}}_{kj}&=& -M_k\bigg(Q_{k+1}B^{-}\left(BM_{k+1}B^-\right)^{\Delta}+P_{k+1}Q_{k+2}B^{-}\left(BM_{k+2}B^-\right)^{\Delta}\\ \\
&&+P_{k+1}P_{k+2}Q_{k+3}B^{-}\left(BM_{k+3}B^-\right)^{\Delta}+\cdots\\ \\
&&+P_{k+1}\ldots P_{j-1}Q_jB^{-}(BM_j B^-)^{\Delta}+ P_{k+1}^{\sigma}\ldots P_jQ_{j+1}B^{-}(BM_{j+1}B^-)^{\Delta}-\cdots\\ \\
&&+P_{k+1}\ldots P_{\nu-2}Q_{\nu-1}B^{-}\left(BM_{\nu-1}B^-\right)^{\Delta}\bigg)B^{\sigma}M_j^{\sigma}B^{-\sigma}B^{\sigma}\\ \\
&=& -M_k\bigg(-Q_{k+1}B^{-}BM_{k+1}B^{-}(BM_j B^-)^{\Delta}-P_{k+1}Q_{k+2}B^{-}BM_{k+2}B^{-}(BM_j B^-)^{\Delta}\\ \\
&&-P_{k+1}P_{k+2}Q_{k+3}B^{-}BM_{k+3}B^{-}(BM_j B^-)^{\Delta}-\cdots\\ \\
&&-P_{k+1}\ldots P_{j-2}Q_{j-1}B^{-}BM_{j-1}B^{-}(BM_j B^-)^{\Delta}-\\ \\
&&- P_{k+1}\ldots P_{j-1}Q_{j}B^{-}BM_{j}B^{-}(BM_j B^-)^{\Delta}\\ \\
&&+P_{k+1}\ldots P_{j-1}Q_j(BM_jB^-)^{\Delta}-P_{k+1}\ldots P_jQ_{j+1}B^{-}BM_{j+1}B^{-}(BM_j B^-)^{\Delta}+\cdots\\ \\
&&-P_{k+1}\ldots P_{\nu-2}Q_{\nu-1}B^{-}BM_{\nu-1}B^{-}(BM_j B^-)^{\Delta}\bigg)B^{\sigma}\\ \\
&=& M_k\bigg(Q_{k+1}+P_{k+1}Q_{k+2}+\cdots+ P_{k+1}\ldots P_{\nu-2}Q_{\nu-1}\\ \\
&&-P_{k+1}\ldots P_{j-1}Q_j\bigg)B^{-}(BM_j B^-)^{\Delta}B^{\sigma}\\ \\
&=& M_k\left(I-P_{k+1}\ldots P_{\nu-1}-P_{k+1}\ldots P_{j-1}Q_j\right)B^{-}(BM_j B^-)^{\Delta}B^{\sigma}.
\end{eqnarray*}
For $j<k$, we get
\begin{eqnarray*}
\widetilde{\mathcal{M}}_{kj}&=& -M_k\bigg(-Q_{k+1}B^{-}BM_{k+1}B^{-}(BM_j B^-)^{\Delta}-P_{k+1}Q_{k+2}B^{-}BM_{k+2}B^{-}(BM_j B^-)^{\Delta}\\ \\
&&-P_{k+1}P_{k+2}Q_{k+3}B^{-}BM_{k+3}B^{-}(BM_j B^-)^{\Delta}\cdots\\ \\
&&-P_{k+1}\ldots P_{\nu-2}Q_{\nu-1}B^{-}BM_{\nu-1}B^{-}(BM_j B^-)^{\Delta}\bigg)B^{\sigma}\\ \\
&=& M_k\bigg(Q_{k+1}+P_{k+1}Q_{k+2}+\cdots+ P_{k+1}\ldots P_{\nu-2}Q_{\nu-1}\bigg)B^{-}(BM_j B^-)^{\Delta}B^{\sigma}\\ \\
&=& M_k\left(I-P_{k+1}\ldots P_{\nu-1}\right)B^{-}(BM_j B^-)^{\Delta}B^{\sigma}.
\end{eqnarray*}
Consequently
\begin{eqnarray*}
J&=& J_1+J_2= \widetilde{\mathcal{K}}_k B^{-\sigma}u^{\sigma}+\sum\limits_{j=1}^{\nu-1}\widetilde{\mathcal{M}}_{kj} M_j^{\sigma} x^{\sigma}
=\widetilde{\mathcal{K}}_k B^{-\sigma}u^{\sigma}+\sum\limits_{j=1}^{\nu-1}\widetilde{\mathcal{M}}_{kj} v_j^{\sigma}
\end{eqnarray*}
and
\begin{eqnarray*}
U_k G_{\nu}^{-}G_0B^{-}(Bx)^{\Delta}&=& \sum\limits_{j=k+1}^{\nu-1}\mathcal{N}_{kj}(Bv_j)^{\Delta}+\widetilde{\mathcal{K}}_k B^{-\sigma}u^{\sigma}+\sum\limits_{j=1}^{\nu-1}\widetilde{\mathcal{M}}_{kj} v_j^{\sigma}.
\end{eqnarray*}

 \subsection{Decomposition of $U_kG_{\nu}^{-1}C^{\sigma}x^{\sigma}$}
\label{section6.6}

In this section, we will find  representations of the terms $U_kG_{\nu}^{-1}C^{\sigma}x^{\sigma}$, using the decomposition
\begin{equation*}
x=\Pi_{\nu-1}x+\sum\limits_{j=0}^{\nu-1}M_j x.
\end{equation*}
We have
\begin{eqnarray*}
U_kG_{\nu}^{-1}C^{\sigma}x^{\sigma}&=& U_kG_{\nu}^{-1}C^{\sigma}\left(\Pi_{\nu-1}^{\sigma}x^{\sigma}+\sum\limits_{j=0}^{\nu-1}M_j^{\sigma} x^{\sigma}\right)\\ \\
&=& U_k G_{\nu}^{-1}C^{\sigma}\Pi_{\nu-1}^{\sigma}x^{\sigma}+\sum\limits_{j=0}^{\nu-1}U_kG_{\nu}^{-1}C^{\sigma}M_j^{\sigma} x^{\sigma}\\ \\
&=& M_kP_{k+1}\ldots P_{\nu-1}G_{\nu}^{-1}C^{\sigma}\Pi_{\nu-1}^{\sigma}x^{\sigma}+\sum\limits_{j=0}^{\nu-1}M_kP_{k+1}\ldots P_{\nu-1}G_{\nu}^{-1}C^{\sigma}M_j^{\sigma} x^{\sigma}.
\end{eqnarray*}
Set
\begin{eqnarray*}
I_1&=& M_kP_{k+1}\ldots P_{\nu-1}G_{\nu}^{-1}C^{\sigma}\Pi_{\nu-1}^{\sigma}x^{\sigma},\\ \\
I_2&=& \sum\limits_{j=0}^{\nu-1}M_kP_{k+1}\ldots P_{\nu-1}G_{\nu}^{-1}C^{\sigma}M_j^{\sigma} x^{\sigma}.
\end{eqnarray*}
Then
\begin{equation}
\label{6.30} U_kG_{\nu}^{-1}C^{\sigma}x^{\sigma}=I_1+I_2.
\end{equation}
Denote
\begin{equation*}
\widehat{\mathcal{K}}_k=M_kP_{k+1}\ldots P_{\nu-1}G_{\nu}^{-1}C^{\sigma}B^{-\sigma}.
\end{equation*}
Hence,
\begin{eqnarray*}
I_1&=& M_kP_{k+1}\ldots P_{\nu-1}G_{\nu}^{-1}C^{\sigma}\Pi_{\nu-1}^{\sigma}x^{\sigma}=  M_kP_{k+1}\ldots P_{\nu-1}G_{\nu}^{-1}C^{\sigma}P_0^{\sigma}\Pi_{\nu-1}^{\sigma}x^{\sigma}\\ \\
&=&  M_kP_{k+1}\ldots P_{\nu-1}G_{\nu}^{-1}C^{\sigma}B^{-\sigma}B^{\sigma}\Pi_{\nu-1}^{\sigma}x^{\sigma}=  M_kP_{k+1}\ldots P_{\nu-1}G_{\nu}^{-1}C^{\sigma}B^{-\sigma}u^{\sigma}= \widehat{\mathcal{K}}_ku^{\sigma},
\end{eqnarray*}
i.e.,
\begin{equation*}
 I_1=\widehat{\mathcal{K}}_ku^{\sigma}.
\end{equation*}

Now, we consider $I_2$. We have
\begin{eqnarray*}
I_2&=&  M_kP_{k+1}\ldots P_{\nu-1}G_{\nu}^{-1}C^{\sigma}M_0^{\sigma}x^{\sigma}+ M_kP_{k+1}\ldots P_{\nu-1}G_{\nu}^{-1}C^{\sigma}\sum\limits_{j=0}^{\nu-1}M_j^{\sigma} x^{\sigma}.
\end{eqnarray*}
Set
\begin{eqnarray*}
\mathcal{M}_{k0}&=&  M_kP_{k+1}\ldots P_{\nu-1}G_{\nu}^{-1}C^{\sigma},\quad 
J= M_kP_{k+1}\ldots P_{\nu-1}G_{\nu}^{-1}C^{\sigma}\sum\limits_{j=0}^{\nu-1}M_j^{\sigma} x^{\sigma}.
\end{eqnarray*}
Hence,
\begin{eqnarray*}
I_2&=& \mathcal{M}_{k0}M_0 ^{\sigma}x^{\sigma}+J= \mathcal{M}_{k0}u^{\sigma}+J.
\end{eqnarray*}
Let us simplify $J$. We have
\begin{eqnarray*}
\lefteqn{ M_kP_{k+1}\ldots P_{\nu-1}G_{\nu}^{-1}C^{\sigma}M_j^{\sigma}}\\ \\
&=&  M_kP_{k+1}\ldots P_{\nu-1}\left(-\left(Q_1+\cdots+Q_j\right)M_j^{\sigma}\right. \\ \\
&-&\left.\sum\limits_{i=1}^j \left(I-Q_i-\cdots - Q_{\nu-1}^ {\sigma}\right)B^{-}(B\Pi_i B^-)^{\Delta}B^{\sigma}M_j^{\sigma}\right).
\end{eqnarray*}

We will consider the following cases.
\begin{enumerate}
\item Let $j<k$. Then
\begin{equation*}
M_kP_{k+1}\ldots P_{\nu-1}\left(Q_1+\cdots+Q_j\right)M_j^{\sigma}=0.
\end{equation*}
Moreover,
\begin{eqnarray*}
\lefteqn{ M_kP_{k+1}\ldots P_{\nu-1}\sum\limits_{i=1}^j \left(I-Q_i-\cdots - Q_{\nu-1}^ {\sigma}\right)B^{-}(B\Pi_i B^-)^{\Delta}B^{\sigma}M_j^{\sigma}}\\ \\
&=&   M_kP_{k+1}\ldots P_{\nu-1}\sum\limits_{i=1}^j \left(I-Q_i-\cdots -Q_{k-1}-Q_k-Q_{k+1}-\cdots - Q_{\nu-1}^ {\sigma}\right)\\ \\
&& B^{-}(B\Pi_i B^-)^{\Delta}B^{\sigma}M_j^{\sigma}\\ \\
&=&  \sum\limits_{i=1}^j \left(M_kP_{k+1}\ldots P_{\nu-1}-M_kP_{k+1}\ldots P_{\nu-1}Q_k\right)B^{-}(B\Pi_i B^-)^{\Delta}B^{\sigma}M_j^{\sigma}\\ \\
&=&  \sum\limits_{i=1}^j \left(M_kP_{k+1}\ldots P_{\nu-1}-M_kP_{k+1}\ldots P_{\nu-1}Q_k\right)B^{-}(B\Pi_i B^-)^{\Delta}B^{\sigma}M_j^{\sigma}\\ \\
&=& \sum\limits_{i=1}^j \left(M_kP_{k+1}\ldots P_{\nu-1}-M_kQ_k\right)B^{-}(B\Pi_i B^-)^{\Delta}B^{\sigma}M_j^{\sigma}\\ \\
&=&  \sum\limits_{i=1}^j\left(M_kP_{k+1}\ldots P_{\nu-1}-M_k\right)B^{-}(B\Pi_i B^-)^{\Delta}B^{\sigma}M_j^{\sigma}\\ \\
&=& \sum\limits_{i=1}^jM_k\left(P_{k+1}\ldots P_{\nu-1}-I\right)B^{-}(B\Pi_i B^-)^{\Delta}B^{\sigma}M_j^{\sigma}.
\end{eqnarray*}
Note that, for $i<j$, we have
\begin{eqnarray*}
\lefteqn{B^{-}(B\Pi_i B^-)^{\Delta}B^{\sigma}M_j^{\sigma}= B^{-}(B\Pi_i B^-)^{\Delta}B^{\sigma}M_j^{\sigma} B^{-\sigma}B^{\sigma}M_j^{\sigma}}\\ \\
&=& B^{-}(B\Pi_i B^- B M_j B^-)^{\Delta}B^{\sigma}M_j^{\sigma}-B^{-}B\Pi_iB^{-}(BM_j B^-)^{\Delta}B^{\sigma}M_j^{\sigma}\\ \\
&=& B^{-}(BM_j B^-)^{\Delta}B^{\sigma}M_j^{\sigma}-\Pi_iB^{-}(BM_j B^-)^{\Delta}B^{\sigma}M_j^{\sigma}= (I-\Pi_i)B^{-}(BM_j B^-)^{\Delta}B^{\sigma}M_j^{\sigma}
\end{eqnarray*}
and 
\begin{eqnarray*}
\lefteqn{\left(P_{k+1}\ldots P_{\nu-1}-I\right)\left(I-\Pi_i\right)}\\ \\
&=& P_{k+1}\ldots P_{\nu-1}-I -P_{k+1}\ldots P_{\nu-1}\Pi_i+\Pi_i\\ \\
&=& P_{k+1} \ldots P_{\nu-1}-I -P_{k+1}\ldots P_{\nu-1}+\\ \\
&+&Q_i+Q_{i-1}P_i+Q_{i-2}P_{i-1}P_i+\cdots +Q_0P_1\ldots P_i+\Pi_i\\ \\
&=& -I+\Pi_i+Q_0P_1\ldots P_i+Q_1P_2\ldots P_i+Q_2P_3\ldots P_i+\cdots+ Q_{i-1}P_i+Q_i\\ \\
&=& -I+P_1P_2\ldots P_i+Q_1P_2\ldots P_i+Q_2P_3\ldots P_i+\cdots+Q_{i-1}P_i+Q_i\\ \\
&=& -I+P_2P_3\ldots P_i+Q_2P_3\ldots P_i+\cdots+ Q_{i-1}P_i+Q_i=\cdots\\ \\
&=& -I+P_{i-1}P_i+Q_{i-1}P_i+Q_i= -I+P_i+Q_i= -I+I= 0.
\end{eqnarray*}
Therefore, for $i<j$, we have
\begin{equation*}
M_k\left(P_{k+1}\ldots P_{\nu-1}-I\right)B^{-}(B\Pi_i B^-)^{\Delta}B^{\sigma}M_j^{\sigma}=0.
\end{equation*}
For $i=j$, we have
\begin{eqnarray*}
&& B^{-}(B\Pi_j B^-)^{\Delta}B^{\sigma}M_j^{\sigma}\\ \\
&=& B^{-}(B\Pi_j B^-)^{\Delta}B^{\sigma}M_j^{\sigma} B^{-\sigma} B^{\sigma}M_j^{\sigma}= B^{-}(B\Pi_j B^- BM_j B^-)^{\Delta}B^{\sigma}M_j^{\sigma}\\ \\
&&- B^{-}B\Pi_jB^{-}(BM_j B^-)^{\Delta}B^{\sigma}M_j^{\sigma}= -\Pi_jB^{-}(BM_j B^-)^{\Delta}B^{\sigma}M_j^{\sigma}
\end{eqnarray*}
and
\begin{eqnarray*}
\lefteqn{(P_{k+1}\ldots P_{\nu-1}-I)\Pi_j=P_{k+1}\ldots P_{\nu-1}\Pi_j-\Pi_j}\\ \\
&=&P_{k+1}\ldots P_{\nu-1}-\Pi_j-Q_0P_1\ldots P_j-Q_1P_2\ldots P_j- \cdots- Q_{j-1}P_j-Q_j\\ \\
&=& P_{k+1}\ldots P_{\nu-1}-P_1\ldots P_j-Q_1P_2\ldots P_j-\cdots - Q_{j-1}P_j-Q_j\\ \\
&=& P_{k+1}\ldots P_{\nu-1}-P_j-Q_j= P_{k+1}\ldots P_{\nu-1}-I.
\end{eqnarray*}
Thus,
\begin{eqnarray*}
&&{M_k\left(P_{k+1}\ldots P_{\nu-1}-I\right)B^{-}(B\Pi_j B^-)^{\Delta}B^{\sigma}M_j^{\sigma}}=\\ \\ 
&&-M_k\left(P_{k+1}\ldots P_{\nu-1}-I\right)B^{-}(BM_jB^-)^{\Delta}B^{\sigma}M_j^{\sigma}
\end{eqnarray*}
and
\begin{equation*}
M_kP_{k+1}\ldots P_{\nu-1}G_{\nu}^{-1}CM_j=-M_k\left(P_{k+1}\ldots P_{\nu-1}-I\right)B^{-}(BM_jB^-)^{\Delta}B^{\sigma}M_j^{\sigma}.
\end{equation*}
Set
\begin{equation*}
\mathcal{M}_{kj}=-M_k\left(P_{k+1}\ldots P_{\nu-1}-I\right)B^{-}(BM_jB^-)^{\Delta}B^{\sigma}.
\end{equation*}
Therefore
\begin{eqnarray*}
M_kP_{k+1}\ldots P_{\nu-1}G_{\nu}^{-1}CM_j&=& \mathcal{M}_{kj} M_j^{\sigma} x^{\sigma}= \mathcal{M}_{kj} v_j^{\sigma}.
\end{eqnarray*}
\item Let $j\geq k$. Then
\begin{eqnarray*}
&&M_kP_{k+1}\ldots P_{\nu-1}\left(Q_1+\cdots+Q_j\right)M_j^{\sigma}\\ \\
&=& M_kP_{k+1}\ldots P_{\nu-1}\left(Q_1+\cdots+Q_{k-1}+Q_k+Q_{k+1}+\cdots+Q_j\right)M_j^{\sigma}\\ \\
&=& M_kP_{k+1}\ldots P_{\nu-1}Q_kM_j+M_kP_{k+1}\ldots P_{\nu-1}Q_{k+1}M_j^{\sigma}+\cdots+M_kP_{k+1}\ldots P_{\nu-1}Q_jM_j^{\sigma}\\ \\
&=& M_kQ_kM_j^{\sigma}= M_kM_j^{\sigma}.
\end{eqnarray*}
Now, using the computations in the previous case, for $j=k$ we get
\begin{eqnarray*}
\lefteqn{M_kP_{k+1}\ldots P_{\nu-1}\sum\limits_{j=1}^k \left(I-Q_i-\cdots-Q_{\nu-1}\right)B^{-}(B\Pi_i B^-)^{\Delta}B^{\sigma}M_k^{\sigma}}\\ \\
&=& M_kP_{k+1}\ldots P_{\nu-1}\sum\limits_{j=1}^{k-1} \left(I-Q_i-\cdots-Q_{\nu-1}\right)B^{-}(B\Pi_i B^-)^{\Delta}B^{\sigma}M_k^{\sigma}\\ \\
&&+M_kP_{k+1}\ldots P_{\nu-1} \left(I-Q_k-\cdots-Q_{\nu-1}\right)B^{-}(B\Pi_i B^-)^{\Delta}B^{\sigma}M_k^{\sigma}\\ \\
&=& M_k(P_{k+1}\ldots P_{\nu-1}-1)\sum\limits_{j=1}^{k-1} B^{-}(B\Pi_i B^-)^{\Delta}B^{\sigma}M_k^{\sigma}\\ \\
&&+M_k(P_{k+1}\ldots P_{\nu-1}-1) B^{-}(B\Pi_i B^-)^{\Delta}B^{\sigma}M_k^{\sigma}\\ \\
&=& M_k(P_{k+1}\ldots P_{\nu-1}-1)\sum\limits_{j=1}^{k-1}(I-\Pi_i) B^{-}(BM_k B^-)^{\Delta}B^{\sigma}M_k^{\sigma}\\ \\
&&-M_k(P_{k+1}\ldots P_{\nu-1}-1) B^{-}(BM_k B^-)^{\Delta}B^{\sigma}M_k^{\sigma}\\ \\
&=& -M_k(P_{k+1}\ldots P_{\nu-1}-1) B^{-}(BM_k B^-)^{\Delta}B^{\sigma}M_k^{\sigma}.
\end{eqnarray*}
Let
\begin{equation*}
\mathcal{M}_{kk}=-M_k(P_{k+1}\ldots P_{\nu-1}-1) B^{-}(BM_k B^-)^{\Delta}B^{\sigma}M_k^{\sigma}-M_k.
\end{equation*}
Let $j>k$. Then
\begin{eqnarray*}
\lefteqn{M_kP_{k+1}\ldots P_{\nu-1}\sum\limits_{i=1}^j \left(I-Q_i-\cdots-Q_{\nu-1}\right)B^{-}(B\Pi_i B^-)^{\Delta}B^{\sigma}M_j^{\sigma}}\\ \\
&=& M_kP_{k+1}\ldots P_{\nu-1}\sum\limits_{i=1}^{k-1} \left(I-Q_i-\cdots-Q_{\nu-1}\right)B^{-}(B\Pi_i B^-)^{\Delta}B^{\sigma}M_j^{\sigma}\\ \\
&&+M_kP_{k+1}\ldots P_{\nu-1} \sum\limits_{i=k}^j\left(I-Q_i-\cdots-Q_{\nu-1}\right)B^{-}(B\Pi_i B^-)^{\Delta}B^{\sigma}M_j^{\sigma}\\ \\
&=& M_k\left(P_{k+1}\ldots P_{\nu-1}-I\right)\sum\limits_{i=1}^{k-1} B^{-}(B\Pi_i B^-)^{\Delta}B^{\sigma}M_j^{\sigma}\\ \\
&&+M_k\left(P_{k+1}\ldots P_{\nu-1}-I\right) \sum\limits_{i=k}^jB^{-}(B\Pi_i B^-)^{\Delta}B^{\sigma}M_j^{\sigma}\\ \\
&=& M_k\left(P_{k+1}\ldots P_{\nu-1}-I\right)\sum\limits_{i=1}^{k-1}(I-\Pi_i) B^{-}(BM_j B^-)^{\Delta}B^{\sigma}M_j^{\sigma}\\ \\
&&+M_k\left(P_{k+1}\ldots P_{\nu-1}-I\right) \sum\limits_{i=k}^j(I-\Pi_i)B^{-}(BM_j B^-)^{\Delta}B^{\sigma}M_j^{\sigma}\\ \\
&&- M_k\left(P_{k+1}\ldots P_{\nu-1}-I\right)\Pi_jB^{-}(BM_jB^-)^{\Delta}B^{\sigma}M_j^{\sigma}\\ \\
&=& M_k\left(P_{k+1}\ldots P_{\nu-1}-I\right) \sum\limits_{i=k}^j(I-\Pi_i)B^{-}(BM_j B^-)^{\Delta}B^{\sigma}M_j^{\sigma}\\ \\
&&- M_k\left(P_{k+1}\ldots P_{\nu-1}-I\right)\Pi_jB^{-}(BM_jB^-)^{\Delta}B^{\sigma}M_j^{\sigma}\\ \\
&=& M_k\left(P_{k+1}\ldots P_{\nu-1}-I\right)(I-\Pi_k)B^{-}(BM_jB^-)^{\Delta}B^{\sigma}M_j^{\sigma}\\ \\
&&+M_k\left(P_{k+1}\ldots P_{\nu-1}-I\right) \sum\limits_{i=k}^{j-1}(I-\Pi_i)B^{-}(BM_j B^-)^{\Delta}B^{\sigma}M_j^{\sigma}\\ \\
&&- M_k\left(P_{k+1}\ldots P_{\nu-1}-I\right)\Pi_jB^{-}(BM_jB^-)^{\Delta}B^{\sigma}M_j^{\sigma}\\ \\
&=& M_k\left(P_{k+1}\ldots P_{\nu-1}-I\right) \sum\limits_{i=k}^{j-1}(I-\Pi_i)B^{-}(BM_j B^-)^{\Delta}B^{\sigma}M_j^{\sigma}\\ \\
&&- M_k\left(P_{k+1}\ldots P_{\nu-1}-I\right)\Pi_jB^{-}(BM_jB^-)^{\Delta}B^{\sigma}M_j^{\sigma}.
\end{eqnarray*}
Note that
\begin{eqnarray*}
\lefteqn{{\left(P_{k+1}\ldots P_{\nu-1}-I\right)\left(I-\Pi_i\right)}= P_{k+1}\ldots P_{\nu-1}-I +\Pi_i-P_{k+1}\ldots P_{\nu-1}\Pi_i}\\ \\
&=& P_{k+1}\ldots P_{\nu-1}-I+\Pi_i-P_{k+1}\ldots P_{\nu-1}+\\ \\
&+& Q_0P_1\ldots P_i+Q_1P_2\ldots P_i+\cdots+Q_kP_{k+1}\ldots P_i\\ \\
&=& -I+ \Pi_i+Q_0P_1\ldots P_{i}+Q_1P_2\ldots P_i+\cdots+Q_kP_{k+1}\ldots P-i\\ \\
&=& -I+P_1\ldots P_i+Q_1P_2\ldots P_i+\cdots+Q_kP_{k+1}\ldots P_i\\ \\
&=& -I+P_2\ldots P_i+\cdots+Q_kP_{k+1}\ldots P_i=\cdots\\ \\
&=& -I+P_kP_{k+1}\ldots P_i+Q_kP_{k+1}\ldots P_i= -I+P_{k+1}\ldots P_i
\end{eqnarray*}
and
\begin{eqnarray*}
\lefteqn{\left(P_{k+1}\ldots P_{\nu-1}-I\right)\Pi_j
= P_{k+1}\ldots P_{\nu-1}-I+\left(P_{k+1}\ldots P_{\nu-1}-I\right)\left(\Pi_j-I\right)}\\ \\
&=& P_{k+1}\ldots P_{\nu-1}-I+I -P_{k+1}\ldots P_j= P_{k+1}\ldots P_{\nu-1}-P_{k+1}\ldots P_j.
\end{eqnarray*}
Consequently
\begin{eqnarray*}
\lefteqn{M_kP_{k+1}\ldots P_{\nu-1}\sum\limits_{i=1}^j \left(I-Q_i-\cdots-Q_{\nu-1}\right)B^{-}(B\Pi_i B^-)^{\Delta}B^{\sigma}M_j^{\sigma}}\\ \\
&=& M_k \sum\limits_{i=k}^{j-1}\left(P_{k+1}\ldots P_{i}-I\right)B^{-}(BM_j B^-)^{\Delta}B^{\sigma}M_j^{\sigma}\\ \\
&&- M_k\left(P_{k+1}\ldots P_{\nu-1}-P_{k+1}\ldots P_j\right)\Pi_jB^{-}(BM_jB^-)^{\Delta}B^{\sigma}M_j^{\sigma}.
\end{eqnarray*}
Then
\begin{eqnarray*}
&&{M_kP_{k+1}\ldots P_{\nu-1}G_{\nu}^{-1}C^{\sigma}M_j^{\sigma}}\\ \\
&=& -M_kM_j^{\sigma}-M_k \sum\limits_{i=k}^{j-1}\left(P_{k+1}\ldots P_{i}-I\right)B^{-}(BM_j B^-)^{\Delta}B^{\sigma}M_j^{\sigma}\\ \\
&&+M_k\left(P_{k+1}\ldots P_{\nu-1}-P_{k+1}\ldots P_j\right)\Pi_jB^{-}(BM_jB^-)^{\Delta}B^{\sigma}M_j^{\sigma}\\ \\
&=&  \bigg(-M_k-M_k \sum\limits_{i=k}^{j-1}\left(P_{k+1}\ldots P_{i}-I\right)B^{-}(BM_j B^-)^{\Delta}B^{\sigma}\\ \\
&&+M_k\left(P_{k+1}\ldots P_{\nu-1}-P_{k+1}\ldots P_j\right)\Pi_jB^{-}(BM_jB^-)^{\Delta}B^{\sigma}\bigg)M_j^{\sigma}.
\end{eqnarray*}
Set
\begin{eqnarray*}
\mathcal{M}_{kj}&=& -M_k-M_k \sum\limits_{i=k}^{j-1}\left(P_{k+1}\ldots P_{i}-I\right)B^{-}(BM_j B^-)^{\Delta}B^{\sigma}\\ \\
&&+M_k\left(P_{k+1}\ldots P_{\nu-1}-P_{k+1}\ldots P_j\right)\Pi_jB^{-}(BM_jB^-)^{\Delta}B^{\sigma}.
\end{eqnarray*}
Therefore
\begin{equation*}
M_kP_{k+1}\ldots P_{\nu-1}G_{\nu}^{-1}C^{\sigma}M_j^{\sigma}= \mathcal{M}_{kj}v_j^{\sigma}.
\end{equation*}
From here,
\begin{equation*}
I_2= \mathcal{M}_{k0}v_0^{\sigma}+\sum\limits_{j=1}^{\nu-1}\mathcal{M}_{kj}v_j^{\sigma}
\end{equation*}
and
\begin{equation}
\label{6.31}
\begin{array}{lll}
U_kG_{\nu}^{-1}C_{\nu}&=& I_1+I_2= \widehat{\mathcal{K}}_ku^{\sigma}+\sum\limits_{j=1}^{\nu-1}\mathcal{M}_{kj}v_j^{\sigma}.
\end{array}
\end{equation}
\end{enumerate}
We multiply the equation \eqref{6.11} by $U_k$ and we find
\begin{equation*}
U_kG_{\nu}^{-1}G_0B^{-}(Bx)^{\Delta}= U_kG_{\nu}^{-1}C^{\sigma}x^{\sigma}+U_kG_{\nu}^{-1}f.
\end{equation*}
Hence, using \eqref{6.30} and \eqref{6.31}, we arrive at
\begin{eqnarray*}
{\sum\limits_{j=k+1}^{\nu-1}\mathcal{N}_{kj}(Bv_j)^{\Delta}+\widetilde{K}_k B^{-\sigma} u^{\sigma}+\sum\limits_{j=0}^{\nu-1}\widetilde{\mathcal{M}}_{kj}v_j^{\sigma}}= \widehat{\mathcal{K}}_k u^{\sigma}+\sum\limits_{j=0}^{\nu-1}\mathcal{M}_{kj}v_j^{\sigma}+U_kG_{\nu}^{-1}f,
\end{eqnarray*}
whereupon
\begin{eqnarray*}
\sum\limits_{j=k+1}^{\nu-1}\mathcal{N}_{kj}(Bv_j)^{\Delta}+\left(\widetilde{\mathcal{K}}_kB^--\widehat{\mathcal{K}}_k\right)+\sum\limits_{j=0}^{\nu-1}\left(\widetilde{\mathcal{M}}_{kj}-\mathcal{M}_{kj}\right)v_j^{\sigma}+ U_kG_{\nu}^{-1}f.
\end{eqnarray*}
Note that
\begin{eqnarray*}
\widetilde{\mathcal{M}}_{kj}-\mathcal{M}_{kj}&=& M_k\left(I- P_{k+1}\ldots P_{\nu-1}\right)B^{-}(BM_jB^-)^{\Delta}B^{\sigma}\\ \\
&&+M_k\left(P_{k+1}\ldots P_{\nu-1}-I\right)B^{-}(BM_j B^-)^{\Delta}B^{\sigma}= 0
\end{eqnarray*}
and
\begin{eqnarray*}
\widetilde{\mathcal{M}}_{kk}-\mathcal{M}_{kk}&=& M_k\left(I- P_{k+1}\ldots P_{\nu-1}\right)B^{-}(BM_kB^-)^{\Delta}B^{\sigma}\\ \\
&&+M_k\left(P_{k+1}\ldots P_{\nu-1}-I\right)B^{-}(BM_k B^-)^{\Delta}B^{\sigma}+M_k= M_k.
\end{eqnarray*}
Therefore
\begin{equation}
\label{6.35}
\begin{array}{lll}
M_k v_k^{\sigma}&=&\sum\limits_{j=k+1}^{\nu-1}\mathcal{N}_{kj}(Bv_j)^{\Delta}+\left(\widetilde{\mathcal{K}}_kB^--\widehat{\mathcal{K}}_k\right)+\sum\limits_{j=k+1}^{\nu-1}\left(\widetilde{\mathcal{M}}_{kj}-\mathcal{M}_{kj}\right)v_j^{\sigma}+ U_kG_{\nu}^{-1}f,\\ \\
v_k&=& M_kv_k.
\end{array}
\end{equation}
Since
\begin{equation*}
BM_k x= BM_k B^-Bx
\end{equation*}
and
\begin{equation*}
BM_k B^-= B\Pi_{k-1}B^--B\Pi_k B^-,
\end{equation*}
and $BM_k xB^-$ and $Bx$ are in $\mathcal{C}^1$, we have that
\begin{equation*}
Bv_k= BM_k x
\end{equation*}
is $\mathcal{C}^1$ for any $k\geq 1$.

\subsection{Decoupling}
In this section, we  will use the notations from the previous sections in this chapter.
\begin{theorem}
\label{theorem6.36} Assume that the pair $(A,B)$ in equation \eqref{6.1} is $(\sigma,1)$-regular with tractability index $\nu$ on $I$ and $f$ is enough smooth function. Then $x\in \mathcal{C}_B^1(I)$ solves \eqref{6.1} if and only if it can be written as
\begin{equation}
\label{6.37}
x=B^- u+v_{\nu-1}+\cdots+v_1+v_0,
\end{equation}
where $u\in \mathcal{C}^1(I)$ solves the inherent equation \eqref{6.17} and $v_k\in \mathcal{C}^1(I)$ satisfies \eqref{6.35}.
\end{theorem}
 \begin{proof}
 If $x\in \mathcal{C}^1(I)$ solves \eqref{6.1}, then  by the computations in the previous sections, we get \eqref{6.17} and \eqref{6.35}. Now, we will prove the converse assertion.  Note that the identity
 \begin{eqnarray*}
 Bv_0&=& BM_0 v_0= 0
 \end{eqnarray*}
 implies that
 \begin{eqnarray*}
 Bx&=& BB^- u+Bv_{\nu-1}+\cdots+ Bv_1= u+Bv_{\nu-1}+\cdots+Bv_1\in \mathcal{C}^1(I),
 \end{eqnarray*}
 where we have used that $u\in \text{im}\,B\Pi_{\nu-1}B^-$, i.e.,
 \begin{equation*}
 u= B\Pi_{\nu-1}B^- u
 \end{equation*}
 and
 \begin{eqnarray*}
 BB^-u&=& BB^-B\Pi_{\nu-1}B^-u= BP_0\Pi_{\nu-1}B^-u= B\Pi_{\nu-1}B^- u= u.
 \end{eqnarray*}
 Now, using \eqref{3.96}, \eqref{3.100}, \eqref{3.102} and the decomposition \eqref{6.37}, we find
 \begin{eqnarray*}
 B\Pi_{\nu-1}x&=& B\Pi_{\nu-1}B^- u+B\Pi_{\nu-1}v_{\nu-1}+\cdots+ B\Pi_{\nu-1}v_1+B\Pi_{\nu-1}v_0\\ \\
 &=& u+B\Pi_{\nu-1}M_{\nu-1}v_{\nu-1}+\cdots+ B\Pi_{\nu-1}M_1v_1+B\Pi_{\nu-1}M_0 v_0=u,
 \end{eqnarray*}
 i.e.,
 \begin{equation*}
 u=B\Pi_{\nu-1}x.
 \end{equation*}
 Now, we multiply \eqref{6.37} by $M_k$ and we obtain
 \begin{eqnarray*}
 M_k x&=& M_k B^- u+M_k v_{\nu-1}+\cdots +M_k v_k +\cdots +M_k v_1+M_k v_0\\ \\
 &=& M_k B^- u+ M_k M_{\nu-1}v_{\nu-1}+\cdots+ M_k M_k v_k+\cdots+M_k M_1v_1+M_k M_0 v_0\\ \\
 &=& M_k B^- u+v_k= M_k B^- B\Pi_{\nu-1}B^- u+v_k= M_k \Pi_{\nu-1}B^- u+v_k= v_k.
 \end{eqnarray*}
 Note that the inherent equation \eqref{6.17} is restated by the equation \eqref{6.11}. Then, we multiply \eqref{6.11} by $B\Pi_{\nu-1}$ and we find
 \begin{equation*}
 B\Pi_{\nu-1}G_{\nu}^{-1}G_0B^{-}(Bx)^{\Delta}=B\Pi_{\nu-1}G_{\nu}^{-1}C^{\sigma}x^{\sigma}+B\Pi_{\nu-1}G_{\nu}^{-1}f,
 \end{equation*}
 which we premultiply by $B^{-}$ and  using that
 \begin{equation*}
 B^-B\Pi_{\nu-1}=\Pi_{\nu-1},
 \end{equation*}
 we find
 \begin{equation*}
 B^{-}B\Pi_{\nu-1}G_{\nu}^{-1}G_0B^{-}(Bx)^{\Delta}=B^{-}B\Pi_{\nu-1}G_{\nu}^{-1}C^{\sigma}x^{\sigma}+B^{-}B\Pi_{\nu-1}G_{\nu}^{-1}f,
 \end{equation*}
 or
 \begin{equation}
 \label{6.38} \Pi_{\nu-1}G_{\nu}^{-1}G_0B^{-}(Bx)^{\Delta}= \Pi_{\nu-1}G_{\nu}^{-1}C^{\sigma}x^{\sigma}+ \Pi_{\nu-1}G_{\nu}^{-1}f,
 \end{equation}
 from where, using \eqref{6.35} and the computations in the previous sections, we get
 \begin{equation}
 \label{6.39} U_kG_{\nu}^{-1}A(Bx)^{\Delta}= U_kG_{\nu}^{-1}C^{\sigma}x^{\sigma}+ U_kG_{\nu}^{-1}f,\quad k=\nu-1, \ldots, 0.
 \end{equation}
 Since
 \begin{equation*}
 Q_k M_k=Q_k,
 \end{equation*}
 we obtain
 \begin{eqnarray*}
 Q_k U_k&=& Q_k M_k P_{k+1}\ldots P_{\nu-1}= Q_k P_{k+1}\ldots P_{\nu-1}= V_k.
 \end{eqnarray*}
 We multiply \eqref{6.39}  by $Q_k$  and we find
 \begin{equation*}
 Q_kU_kG_{\nu}^{-1}A(Bx)^{\Delta}= Q_kU_kG_{\nu}^{-1}C^{\sigma}x^{\sigma}+ Q_kU_kG_{\nu}^{-1}f,\quad k=\nu-1, \ldots, 0,
 \end{equation*}
 or
 \begin{equation}
 \label{6.40} V_kG_{\nu}^{-1}A^{\sigma}(Bx)^{\Delta}= V_kG_{\nu}^{-1}C^{\sigma}x^{\sigma}+ V_kG_{\nu}^{-1}f,\quad k=\nu-1, \ldots, 0.
 \end{equation}
 Note that
 \begin{equation*}
 I=\Pi_{\nu-1}+\sum\limits_{k=0}^{\nu-1}V_k.
 \end{equation*}
 Then, by \eqref{6.38} and \eqref{6.40}, we get
 \begin{equation*}
 \left(\Pi_{\nu-1}+\sum\limits_{k=0}^{\nu-1}V_k\right)G_{\nu}^{-1}A^{\sigma}(Bx)^{\Delta}= \left(\Pi_{\nu-1}+\sum\limits_{k=0}^{\nu-1}V_k\right)G_{\nu}^{-1}C^{\sigma}x^{\sigma}+ \left(\Pi_{\nu-1}+\sum\limits_{k=0}^{\nu-1}V_k\right)G_{\nu}^{-1}f
 \end{equation*}
 or
 \begin{equation*}
 G_{\nu}^{-1}A^{\sigma}(Bx)^{\Delta}=G_{\nu}^{-1}C^{\sigma}x^{\sigma}+G_{\nu}^{-1}f,
 \end{equation*}
 or
 \begin{equation*}
 A^{\sigma}(Bx)^{\Delta}=C^{\sigma}x^{\sigma}+f.
 \end{equation*}
 This completes the proof.
 \end{proof}
 \section{An Example}
  Let $\mathbb{T}=2^{\mathbb{N}_0}$ and
 \begin{eqnarray*}
 A(t)&=& \left(
           \begin{array}{ccc}
             1 & 0 & 0 \\
             0 & \frac{1}{t} & 0 \\
             0 & 0 & 1 \\
             0 & 0 & 0 \\
             0 & 0 & 0 \\
           \end{array}
         \right),\quad B(t)= \left(
           \begin{array}{ccccc}
             1 & 0 & 0 & 0 & 0 \\
             0 & 2t & 0 & 0 & 0 \\
             0 & 0 & 1 & 0 & 0 \\
           \end{array}
         \right),
          \end{eqnarray*}
           \begin{eqnarray*}
         C(t)= \left(
           \begin{array}{ccccc}
             0 & 0 & 0 & -1 & 1 \\
             0 & 0 & 1 & 1 & 0 \\
             0 & -1 & 0 & 0 & 0 \\
             -1 & 1 & 0 & 0 & 0 \\
             1 & 0 & 0 & 0 & t^2 \\
           \end{array}
         \right),\quad t\in \mathbb{T}.
 \end{eqnarray*}
We have
 $\sigma(t)=2t$, $t\in \mathbb{T}$, 
 and
 \begin{eqnarray*}
 A^{\sigma}(t)&=& \left(
           \begin{array}{ccc}
             1 & 0 & 0 \\
             0 & \frac{1}{2t} & 0 \\
             0 & 0 & 1 \\
             0 & 0 & 0 \\
             0 & 0 & 0 \\
           \end{array}
         \right),\quad 
 G_0(t)= A^{\sigma}(t)B(t)= \left(
       \begin{array}{ccccc}
         1 & 0 & 0 & 0 & 0 \\
         0 & 1 & 0 & 0 & 0 \\
         0 & 0 & 1 & 0 & 0 \\
         0 & 0 & 0 & 0 & 0 \\
         0 & 0 & 0 & 0 & 0 \\
       \end{array}
     \right),\quad t\in \mathbb{T}.
 \\ \\
 Q_0(t)&=& M_0(t)= \left(
          \begin{array}{ccccc}
            0 & 0 & 0 & 0 & 0 \\
            0 & 0 & 0 & 0 & 0 \\
            0 & 0 & 0 & 0 & 0 \\
            0 & 0 & 0 & 1 & 0 \\
            0 & 0 & 0 & 0 & 1 \\
          \end{array}
        \right),\\ \\
 P_0(t)&=& \Pi_0(t)= I-M_0(t)= \left(
       \begin{array}{ccccc}
         1 & 0 & 0 & 0 & 0 \\
         0 & 1 & 0 & 0 & 0 \\
         0 & 0 & 1 & 0 & 0 \\
         0 & 0 & 0 & 0 & 0 \\
         0 & 0 & 0 & 0 & 0 \\
       \end{array}
     \right),\quad t\in \mathbb{T}.
 \end{eqnarray*}
 Next,
 \begin{eqnarray*}
 G_1(t)&=& G_0(t)+C_0(t)Q_0(t)= \left(
           \begin{array}{ccccc}
             1 & 0 & 0 & -1 & 1 \\
             0 & 1 & 0 & 1 & 0 \\
             0 & 0 & 1 & 0 & 0 \\
             0 & 0 & 0 & 0 & 0 \\
             0 & 0 & 0 & 0 & t^2 \\
           \end{array}
         \right),\\ \\
 Q_1(t)&=& \left(
           \begin{array}{ccccc}
             1 & 0 & 0 & 0 & 0 \\
             -1 & 0 & 0 & 0 & 0 \\
             0 & 0 & 0 & 0 & 0 \\
             1 & 0 & 0 & 0 & 0 \\
             0 & 0 & 0 & 0 & 0 \\
           \end{array}
         \right),\quad 
 P_1(t)= \left(
           \begin{array}{ccccc}
             0 & 0 & 0 & 0 & 0 \\
             1 & 1 & 0 & 0 & 0 \\
             0 & 0 & 1 & 0 & 0 \\
             -1 & 0 & 0 & 1 & 0 \\
             0 & 0 & 0 & 0 & 1 \\
           \end{array}
         \right), \\ \\
 \Pi_1(t)&=& P_0(t)P_1(t)= \left(
           \begin{array}{ccccc}
             0 & 0 & 0 & 0 & 0 \\
             1 & 1 & 0 & 0 & 0 \\
             0 & 0 & 1 & 0 & 0 \\
             0 & 0 & 0 & 0 & 0 \\
             0 & 0 & 0 & 0 & 0 \\
           \end{array}
         \right),\quad 
B^-(t)=\left(
         \begin{array}{ccc}
           1 & 0 & 0 \\
           0 & \frac{1}{2t} & 0 \\
           0 & 0 & 1 \\
           0 & 0 & 0 \\
           0 & 0 & 0 \\
         \end{array}
       \right),\\ \\
B(t) \Pi_0(t) B^-(t)
&=& \left(
      \begin{array}{ccc}
        1 & 0 & 0 \\
        0 & 1 & 0 \\
        0 & 0 & 1 \\
      \end{array}
    \right),\quad t\in \mathbb{T}.
\end{eqnarray*}
Therefore
\begin{equation*}
(B\Pi_0 B^-)^{\Delta}(t)=0,\quad t\in \mathbb{T}.
\end{equation*}
Next,
\begin{eqnarray*}
M_1(t)&=& \Pi_0(t)-\Pi_1(t)
= \left(
           \begin{array}{ccccc}
             1 & 0 & 0 & 0 & 0 \\
             -1 & 0 & 0 & 0 & 0 \\
             0 & 0 & 0 & 0 & 0 \\
             0 & 0 & 0 & 0 & 0 \\
             0 & 0 & 0 & 0 & 0 \\
           \end{array}
         \right),\\ \\
B(t)\Pi_1(t) B^-(t)&=& \left(
      \begin{array}{ccc}
        0 & 0 & 0 \\
        1 & 1 & 0 \\
        0 & 0 & 1 \\
      \end{array}
    \right),\quad t\in \mathbb{T},
\end{eqnarray*}
whereupon
\begin{equation*}
(B\Pi_1 B^-)^{\Delta}(t)=0,\quad t\in \mathbb{T}.
\end{equation*}
Moreover,
\begin{eqnarray*}
C_1(t)&=& \left(
            \begin{array}{ccccc}
              0 & 0 & 0 & 0 & 0 \\
              0 & 0 & 1 & 0 & 0 \\
              -2 & -1 & 0 & 0 & 0 \\
              3 & 1 & 0 & 0 & 0 \\
              -1 & 0 & 0 & 0 & 0 \\
            \end{array}
          \right), \quad 
C_1(t)M_1(t)
= \left(
           \begin{array}{ccccc}
             0 & 0 & 0 & 0 & 0 \\
             0 & 0 & 0 & 0 & 0 \\
             -1 & 0 & 0 & 0 & 0 \\
             2 & 0 & 0 & 0 & 0 \\
             -1 & 0 & 0 & 0 & 0 \\
           \end{array}
         \right),\quad t\in \mathbb{T},
\end{eqnarray*}
and
\begin{eqnarray*}
G_2(t)&=& G_1(t)+C_1(t) Q_1(t)= \left(
           \begin{array}{ccccc}
             1 & 0 & 0 & -1 & 1 \\
             0 & 1 & 0 & 1 & 0 \\
             -1 & 0 & 1 & 0 & 0 \\
             2 & 0 & 0 & 0 & 0 \\
             -1 & 0 & 0 & 0 & t^2 \\
           \end{array}
         \right),\quad t\in \mathbb{T}.
\end{eqnarray*}
We compute
\begin{eqnarray*}
\det G_0(t)=\det G_1(t)=0,\quad \det G_2(t)=2t^2\ne 0,\quad t\in \mathbb{T}.
\end{eqnarray*}
Thus, $(A, B)$ is a $(\sigma, 1)$-regular matrix pair with tractability index $2$.

\section{Conclusion}

Decoupling in dynamic systems is a powerful tool that allows to simplify a system and, sometimes, solves it analytically. Besides, this is a very important approach in stability theory where it allows the construction of bounded solutions, Perron manifolds, etc. We plan to use the obtained result in developing a specific structural stability theory for time-scale systems.

There are several stumbling blocks in this way.

\begin{enumerate}
\item Some linear time scale systems have non-unique solutions. 
\item The role of autonomous systems in time scale dynamics is much less pronounced than that in ordinary differential equations. In particular, exponential functions are not always a convenient tool to estimate the growth of solutions. Moreover, it is not easy to construct such a scale of functions (e.g. to introduce Lyapunov exponents) for a generic time scale. 
\item The Functional Analysis on time scales is less developed and the statements are not always similar to classical ones.
\end{enumerate}

This is why we developed a specific version of the projection approach that does not require us to make additional estimates.

We formulate the main statement of our paper as Theorem \ref{theorem6.36}. There, we formulate the conditions for decoupling. Moreover, we offer an explicit decoupling procedure based. We mustn't make any assumptions about the asymptotic behavior of solutions (e.g. exponential dichotomy). 

\section*{Acknowledgements}
The work of the second co-author was supported by Gda\'{n}sk University of Technology by the DEC 14/2021/IDUB/I.1 grant under the Nobelium - 'Excellence Initiative - Research University' program.

\end{document}